\documentclass[a4paper,10pt]{article}
\usepackage[dvipsnames]{xcolor}
\usepackage{amsmath}
\usepackage{amsfonts}
\usepackage{amssymb}
\usepackage{amsthm}
\usepackage{enumerate}
\usepackage{hyperref}
\usepackage{graphicx}
\usepackage{pgfplots}
\usepackage{tikz}
\usepackage{mathrsfs}
\usetikzlibrary{external}
\usepgfplotslibrary{groupplots}
\usepackage{svg}

\usepackage{authblk}

\usepackage{cleveref}
\usepackage{a4wide}

\newtheorem{theorem}{Theorem}[section]
\newtheorem{proposition}[theorem]{Proposition}
\newtheorem{lemma}[theorem]{Lemma}
\newtheorem{corollary}[theorem]{Corollary}
\newtheorem{conjecture}[theorem]{Conjecture}
\theoremstyle{definition}

\theoremstyle{remark}
\newtheorem{remark}[theorem]{Remark}

\renewcommand{\d}{\mathrm{d}}
\newcommand{\eps}{\varepsilon}

\RequirePackage{etoolbox}
\newcommand{\DeclareMyMathOperator}[1]{%
	\expandafter\DeclareMathOperator\csname my#1\endcsname{#1}
}
\newcommand{\DeclareMathOperators}{\forcsvlist{\DeclareMyMathOperator}}
\DeclareMathOperators{supp,diag, span, tr, modulo,sign, Re}

\def\Hspace{\mathcal H}
\def\PN{P_N}

\def\EDo{-1,1}

\def\T{T}

\newcommand{\per}{\mathrm{per}}

\title{Pointwise and uniform convergence of Fourier extensions}

\author[1]{Marcus Webb\footnote{Email: \texttt{marcus.webb@manchester.ac.uk}. Website: \texttt{https://personalpages.manchester.ac.uk/staff/marcus.webb/}.}}
\author[2]{Vincent Copp\'e\footnote{Email: \texttt{vincent.coppe@cs.kuleuven.be}. Website: \texttt{http://people.cs.kuleuven.be/\textasciitilde vincent.coppe}.}}
\author[2]{Daan Huybrechs\footnote{Email: \texttt{daan.huybrechs@cs.kuleuven.be}. Website: \texttt{http://people.cs.kuleuven.be/\textasciitilde daan.huybrechs}.}}
\affil[1]{Department of Mathematics, University of Manchester, United Kingdom}
\affil[2]{Department of Computer Science, KU Leuven, Belgium}

\date{\today}

\begin{document}

\maketitle

\begin{abstract}
Fourier series approximations of continuous but nonperiodic functions on an interval suffer the Gibbs phenomenon, which means there is a permanent oscillatory overshoot in the neighbourhoods of the endpoints. Fourier extensions circumvent this issue by approximating the function using a Fourier series which is periodic on a larger interval. Previous results on the convergence of Fourier extensions have focused on the error in the $L^2$ norm, but in this paper we analyze pointwise and uniform convergence of Fourier extensions (formulated as the best approximation in the $L^2$-norm). We show that the pointwise convergence of Fourier extensions is more similar to Legendre series than classical Fourier series. In particular, unlike classical Fourier series, Fourier extensions yield pointwise convergence at the endpoints of the interval. Similar to Legendre series, pointwise convergence at the endpoints is slower by an algebraic order of a half compared to that in the interior. The proof is conducted by an analysis of the associated Lebesgue function, and Jackson- and Bernstein-type theorems for Fourier extensions. Numerical experiments are provided. We conclude the paper with open questions regarding the regularized and oversampled least squares interpolation versions of Fourier extensions.
\end{abstract}

 \noindent \textbf{Keywords} \, Fourier extension, Lebesgue function, Legendre polynomials on a circular arc, constructive approximation \\

 \noindent \textbf{Mathematics Subject Classification (2010)} \, 42A10, 41A17, 65T40, 42C15

\section{Introduction}\label{sec:intro}

The Fourier series of a periodic function converges spectrally fast with respect to the number of terms in the series, that is, with an algebraic order which increases with the number of available derivatives and exponentially fast for analytic functions. Furthermore, the truncated Fourier series can be approximated via the FFT in a fast and stable manner \cite{wright2015extension}. As such, it is the go-to approach to approximate a periodic function. However, when the function in question is nonperiodic, the situation is very different. Regardless of how smooth this function is, convergence is slow in the $L^2$ norm and there is a permanent oscillatory overshoot close to the endpoints due to the Gibbs phenomenon \cite{zygmund2002trigonometric}.

Fourier extensions have been shown to be an effective means for the approximation of nonperiodic functions while avoiding the Gibbs phenomenon \cite{adcock2014resolution,adcock2014numerical,boyd2002comparison,huybrechs2010fourier,lyon2011fast,lyon2012approximation,matthysen2016fast}. The idea is as follows: For a function $f \in L^2(-1,1)$, consider an approximant $f_N$ given by
\begin{equation}\label{eqn:FE1}
f_N(x) = \sum_{k=-n}^n c_k e^{\frac{i\pi}{\T}kx}, \qquad N = 2n+1,
\end{equation}
where the coefficients $c_{-n},\ldots c_n$ are chosen to minimise the error $\|f - f_N\|_{L^2(-1,1)}$, and $T > 1$ is a user-determined parameter. This approximant $f_N$ is the $n$th Fourier extension of $f$ to the periodic interval $[-T,T]$. For the purposes of this paper, others kinds of Fourier extension, which might come from a discrete sampling of $f$ or regularization, are a modification of this\footnote{Articles such as \cite{adcock2014numerical} refer to this type of Fourier extension as the \emph{exact continuous} Fourier extension.}.

There are many approximation schemes that avoid the Gibbs phenomenon. Chebyshev polynomial interpolants such as those implemented in the Chebfun \cite{driscoll2014chebfun,trefethen2013approximation} and ApproxFun \cite{olver2018approxfun} software packages are extremely successful, so why consider Fourier extensions? First, discrete collocation versions of Fourier extensions sample the function on equispaced or near-equispaced grids, which in some situations are more natural than Chebyshev grids, which cluster near the endpoints \cite{adcock2014parameter}. Second, the approach generalises naturally to higher dimensions. If one has a function on a bounded subset $\Omega \subset \mathbb{R}^d$, then one can use multivariate Fourier series which are periodic on a $d$-dimensional bounding box containing $\Omega$ \cite{boyd2005fourier,huybrechs2016computing,matthysen2018function}. Modifications of Fourier extensions which use discete samples of a function are particularly relevant in this generalization, because the integrals defining the $L^2(\Omega)$ norm can be difficult to compute.

Fourier extensions can be computed stably in $\mathcal{O}(N\log^2(N))$ floating point operations, with the following important caveats (\cite{karnik2017fast,matthysen2016fast,lyon2011fast}). Computation of $f_N$ is equivalent to inversion of the so-called prolate matrix \cite{varah1993prolate}, which is a Toeplitz matrix $G \in \mathbb{R}^{N\times N}$ with entries $G_{k,j} = \mathrm{sinc}\left( (k-j)\frac{\pi}{\T}\right)$, with right-hand-side vector $\mathbf{b}\in\mathbb{C}^N$ with entries $b_k = \left(\frac\T2\right)^{\frac12}\int_{-1}^1 e^{-\frac{i\pi}{\T}kx}f(x)\,\d x$ \cite{matthysen2016fast}. The prolate matrix is exponentially ill-conditioned \cite[Eq.~63]{slepian1978prolate}, so computation of the exact Fourier extension is practically impossible, even for moderately sized $N$. However, a truncated Singular Value Decomposition (SVD) solution is only worse than the exact solution (in the $L^2(-1,1)$ norm) by a small factor $\mathcal{O}(\varepsilon^{\frac12})$ in the limit as $N \to \infty$, where $\varepsilon > 0$ is the truncation parameter \cite{adcock2014numerical,adcock2018frames}. Furthermore, using an oversampled least squares interpolation in equispaced points in $[-1,1]$ can bring this down to $\mathcal{O}(\varepsilon)$ for a sufficient oversampling rate \cite{adcock2014numerical,adcock2018frames,adcock2017frames}. At the heart of these facts is the observation that while the Fourier basis on $[-T,T]$ does not form a Schauder basis for $L^2(-1,1)$, it satisfies the weaker conditions of a \emph{frame} \cite{adcock2018frames}.
 
Fourier extensions which approximate a truncated SVD solution rather than the exact solution are called \emph{regularized} Fourier extensions. An approximate SVD of the prolate matrix can be computed in $\mathcal{O}(N\log^2(N))$ operations using the Fast Fourier Transform (FFT) and exploiting the so-called plunge region in the profile of its singular values \cite{karnik2017fast}. This is a vast improvement on $\mathcal{O}(N^3)$ operations for a standard SVD. Fast algorithms for regularized, oversampled least squares interpolation Fourier extensions were developed in \cite{matthysen2016fast}, building on the work of Lyon \cite{lyon2011fast}. 
 
Previous convergence results on Fourier extensions have focused on convergence in the $L^2$ norm, because the Fourier extension by definition minimizes the error in the $L^2$ norm over the approximation space. Convergence in $L^2$ of algebraic order $k$ for functions in the Sobolev space $H^k(-1,1)$ was proved by Adcock and Huybrechs \cite[Thm.~2.1]{adcock2014numerical}. It follows immediately that convergence is superalgebraic for smooth functions. Exponential convergence in $L^2$ and $L^\infty$ norms for analytic functions was proved by Huybrechs for $T =2$ \cite{huybrechs2010fourier} and by Adcock et al. for general $T > 1$ \cite{adcock2014numerical}. The proofs of exponential convergence appeal to connections between the Fourier extension problem and the sub-range Chebyshev polynomials \cite{adcock2014numerical}, for which series approximations converge at an exponential rate which depends on analyticity in Bernstein ellipses in the complex plane. Regarding pointwise convergence of Fourier extensions for non-analytic functions, there are no proofs in the literature. Some numerical exploration of pointwise convergence appears in \cite[Sec.~2]{bruno2007accurate}, but a rigorous theoretical foundation is lacking.

\subsection{Summary of new results}

In this paper we prove that for $f$ in the H\"older space $C^{k,\alpha}([-1,1])$,
\begin{equation}\label{eqn:mainresult}
f(x) - f_N(x) = \begin{cases}
\mathcal{O}(N^{-k-\alpha}\log(N)) & \text{ for } x \in [a,b] \subset (-1,1) \\
\mathcal{O}(N^{\frac12-k-\alpha})) & \text{ for } x \in [-1,1],
\end{cases}
\end{equation}
see Theorem \ref{thm:Linftyalgebraic}. The factors of $\log(N)$ and $N^{\frac12}$ come from bounds on the Lebesgue function associated with Fourier extension derived in Section \ref{sec:prolatekernel}, and the factor of $N^{-k-\alpha}$ comes from a Jackson-type theorem proved for Fourier extensions derived in Section \ref{sec:bestuniform} on best uniform approximation by Fourier extensions. 

This factor of $N^{-k-\alpha}$ can be pessimistic if $f$ is least regular at the boundary; in Section \ref{sec:bestuniform} we discuss how a weighted form of regularity (as opposed to H\"older regularity taken uniformly over the interval $[-1,1]$) might yield a more natural correspondence between regularity and convergence rate. This is precisely the case in best polynomial approximation on an interval, where weighted moduli of continuity have a tight correspondence with best approximation errors \cite[Ch.~7, Thm.~7.7]{devore1993constructive}.

From equation \eqref{eqn:mainresult}, it is immediate that if $f \in C^{\alpha}([-1,1])$ where $\alpha \in (0,1)$, then $f_N$ converges to $f$ uniformly in any subinterval $[a,b] \subset (-1,1)$, and if $\alpha > \frac12$, then we get uniform convergence over the whole interval $[-1,1]$. 

We also prove a local pointwise convergence result, which states that if $f \in L^2(-1,1)$, but $f$ is uniformly Dini--Lipschitz in a subinterval $[a,b]$, then the Fourier extension converges uniformly in compact subintervals of $(a,b)$ (see Theorem \ref{thm:DiniLipschitz}). This is done by generalizing a localization theorem of Freud on convergence of orthogonal polynomial expansions in $[-1,1]$ (see Section \ref{sec:localization}).
 
 A key insight of this paper is that the kernel associated with approximation by Fourier extension has an explicit formula which is related to the Christoffel-Darboux kernel of the \emph{Legendre polynomials on a circular arc} (see Lemma \ref{lem:prolatekernelformula}). The asymptotics of these polynomials were derived by Krasovksy using Riemann--Hilbert analysis \cite{krasovsky2004gap,kuijlaars2004riemann,deift1999orthogonal}, which we use to derive asymptotics of the kernel. The Lebesgue function for Fourier extensions are estimated using these asymptotics in Theorem \ref{thm:lebesguefunction}. We find that the Lebesgue function is $\mathcal{O}(\log(N))$ in the interior of $[-1,1]$ and $\mathcal{O}(N^{\frac12})$ globally. This is just as with the Lebesgue function for Legendre series, and distinct from classical Fourier series which has a $\mathcal{O}(\log N)$ Lebesgue function over the full periodic interval.
 
The results of this paper would become more interesting when they can be extended to regularized and oversampled interpolation versions of Fourier extensions, because as discussed above, these are the versions for which stable and efficient algorithms have been developed. The multivariate case is another direction this line of inquiry would ideally lead. We briefly discuss future research like this in Section \ref{sec:discussion}.
 
 The paper is structured as follows. Section 2 recounts the known results about convergence of Fourier extensions in the $L^2$ norm. Section 3 gives new pointwise and uniform convergence theorems along with proofs which depend on results proved in the self-contained Sections 4, 5, and 6. Section 4 is on the Lebesgue function for Fourier extensions. Section 5 is on uniform best approximation for Fourier extensions, in which Jackson- and Bernstein-type theorems are proved. Section 6 is on an analogue of Freud's localization theorem for Fourier extensions. Section 7 provides the reader with results from numerical experiments, and Section 8 provides discussion. The appendix contains a derivation of asymptotics of Legendre polynomials on a circular arc, on the arc itself, from the Riemann--Hilbert analysis of Krasovsky \cite{krasovsky2004gap,kuijlaars2004riemann,deift1999orthogonal}.

\section{Convergence of Fourier extensions in $L^2$}\label{sec:L2convergence}
 In this section we summarise the already known results regarding convergence in the $L^2$ norm.

\subsection{Exponential convergence}

As is discussed in \cite{huybrechs2010fourier,adcock2014resolution}, the Fourier extension $f_N$ in equation \eqref{eqn:FE1} is a polynomial in the mapped variable $t = m(x)$, where
\begin{equation*}
m(x) = 2 \frac{\cos\left(\frac{\pi}{\T}x \right) - \cos\left(\frac{\pi}{\T} \right)}{1-\cos\left(\frac{\pi}{\T}\right)} - 1.
\end{equation*}
This change of variables transforms the Fourier extension problem into two series expansions in modified Jacobi polynomials \cite{huybrechs2010fourier}. Since exponential convergence in this setting is dictated by Bernstein ellipses in the complex plane, which are defined to be the closed contours,
\begin{equation*}
\mathcal{B}(\rho) = \left\{  \frac{1}{2}\left( \rho e^{i\theta} + \frac{1}{\rho} e^{-i\theta}\right) :  \theta\in[-\pi,\pi] \right\}, \qquad \rho > 1,
\end{equation*}
it makes sense to consider the mapped contours,
\begin{equation}\label{eqn:Drho}
\mathcal{D}(\rho) := m^{-1} \left( \mathcal{B}(\rho) \right),
\end{equation}
 as a candidate for determining the rate of exponential convergence for Fourier extensions. They are indeed the relevant contours, as was proven in the following theorem.
\begin{theorem}[Adcock--Huybrechs {\cite[Thm.~3.14]{huybrechs2010fourier},\cite[Thm.~2.3]{adcock2014numerical}}]\label{thm:analyticexactsobolev}
  If $f$ is an analytic function in $\mathcal{D(\rho^\star)}$ and continuous on $\mathcal{D(\rho^\star)}$ itself, then
  \begin{equation*}
  \|f-f_N\|_{L^2(\EDo)} = \mathcal{O}(\rho^{-n}) \|f\|_{L^\infty(\mathcal{D}(\rho))},
  \end{equation*}
  where $\rho < \min\left\{\rho^\star, \cot^2\left(\frac{\pi}{4\T}\right)\right\}$ and $N = 2n+1$. The constant in the big $\mathcal{O}$ depends only on $T$.
\end{theorem}
Note that there is a $T$-dependent upper limit on the rate of exponential convergence.

\subsection{Algebraic convergence}
For functions in the Sobolev space $H^k(\EDo)$ of $L^2(-1,1)$ functions whose $k$th weak derivatives are in $L^2(-1,1)$, we have algebraic convergence of order $k$.
\begin{theorem}\label{thm:lsproj}[Adcock--Huybrechs \cite[Thm.~2.1]{adcock2014resolution}]
  If $f\in H^k(\EDo)$, then	
  \begin{equation*}
  \|f-f_N\|_{L^2(\EDo)} = \mathcal{O}(N^{-k}) \|f\|_{H^k(\EDo)},
  \end{equation*}
  where the constant in the big $\mathcal{O}$ depends only on $k$ and $T$.
\end{theorem}

\begin{corollary}\label{cor:L2super}
  If $f$ is smooth then $f_N \to f$ superalgebraically in the $L^2(-1,1)$ norm.
\end{corollary}

\subsection{Subalgebraic convergence}
This elementary result says that Fourier extensions converge in the $L^2$ norm for $L^2$ functions.
\begin{proposition}\label{prop:L2convergence}
  If $f \in L^2(-1,1)$, then
  \begin{equation*}
  \|f - f_N\|_{L^2(-1,1)} \to 0 \text{ as } N \to \infty.
  \end{equation*}
  \begin{proof}
    Let $g \in L^2(-T,T)$ be the function that is equal to $f$ inside $[-1,1]$ and zero in the complement. Let $g(x) = \sum_{k=-\infty}^\infty c_k e^{\frac{i\pi}{\T}kx}$ be its Fourier series, and for all odd integers $N = 2n+1$, define $t_N(x) =  \sum_{k=-n}^n c_k e^{\frac{i\pi}{\T}kx}$. Then following the definitions of $f_N$, $g$ and $t_N$, we have $\|f - f_N \|_{L^2(-1,1)} \leq \|f - t_N \|_{L^2(-1,1)} = \|g - t_N \|_{L^2(-T,T)} \to 0$ as $N\to\infty$.
  \end{proof}
\end{proposition}

\section{Pointwise and uniform convergence}\label{sec:pointwiseconvergence}

We prove pointwise convergence rates for functions in various H\"older spaces. For $k = 0,1,2,\ldots$ and $\alpha \in [0,1]$, the H\"older space $C^{k,\alpha}([-1,1])$ is the space,
\begin{equation*}
C^{k,\alpha}([-1,1]) : = \left\{f \in C^k([-1,1]) : |f^{(k)}|_{C^{\alpha}([-1,1])} < \infty \right\},
\end{equation*}
where
\begin{equation*}
|g|_{C^\alpha([-1,1])} := \sup_{x,y \in [-1,1]} \frac{|g(x)-g(y)|}{|x-y|^\alpha}.
\end{equation*}
It is a Banach space when endowed with the norm $\|f\|_{C^{k,\alpha}([-1,1])} = \|f\|_{C^k([-1,1])} + |f^{(k)}|_{C^\alpha([-1,1])}$ \cite{evans2010partial}. For all $\alpha \in [0,1]$, we have $C^\alpha([-1,1]) := C^{0,\alpha}([-1,1])$.

\subsection{Exponential convergence}

The pointwise convergence result for analytic functions is the same as Theorem \ref{thm:analyticexactsobolev}. In fact, Theorem \ref{thm:analyticexactsobolev} is a corollary of the following theorem.

\begin{theorem}[{Huybrechs \cite[Theorem 3.14]{huybrechs2010fourier}, Adcock--Huybrechs \cite[Theorem 2.11]{adcock2014numerical}, \cite[Theorem 2.3]{adcock2014resolution}}]\label{thm:analyticuniform}
  If $f$ is analytic inside of the mapped Bernstein ellipse, $\mathcal{D(\rho^\star)}$ (see equation \eqref{eqn:Drho}) and continuous on $\mathcal{D(\rho^\star)}$ itself, then
  \begin{equation*}
  \|f-f_N\|_{L^\infty(\EDo)} = \mathcal{O}(\rho^{-n}) \|f\|_{L^\infty(\mathcal{D}(\rho))},
  \end{equation*}
  where $\rho < \min\left\{\rho^\star, \cot^2\left(\frac{\pi}{4\T}\right)\right\}$ and $N= 2n+1$. The constant in the big $\mathcal{O}$ depends only on $T$.
\end{theorem}

\subsection{Algebraic convergence}

Pointwise convergence for H\"older continuous functions is as follows.

\begin{theorem}\label{thm:Linftyalgebraic}

  If $f\in C^{k,\alpha}([-1,1])$ where $k \geq 0$ and $\alpha \in [0,1]$, then for all $[a,b] \subset (-1,1)$,
    \begin{equation*}
    \|f - f_N\|_{L^\infty(a,b)} = \mathcal{O}(N^{- \alpha - k}\log N) |f^{(k)}|_{C^{\alpha}([-1,1])}.
    \end{equation*}
  The constant in the big $\mathcal{O}$ depends on $a$, $b$, $k$, $\alpha$, and $T$. Over the whole interval, $[-1,1]$, we have
  \begin{equation*}
  \|f - f_N\|_{L^\infty(-1,1)} = \mathcal{O}(N^{\frac12 - \alpha - k}) |f^{(k)}|_{C^{\alpha}([-1,1])}.
  \end{equation*}
  The constant in the big $\mathcal{O}$ depends on $k$, $\alpha$, and $T$.
\end{theorem}

We lose a half order of algebraic convergence at the endpoints, something that we could not possibly see in classical Fourier series because a periodic interval has no endpoints.

\begin{corollary}
  If $f$ is smooth then $f_N \to f$ superalgebraically in $L^\infty(-1,1)$.
\end{corollary}

\subsection{Subalgebraic convergence}

The loss of a half order of algebraic convergence at the endpoints predicted by Theorem \ref{thm:Linftyalgebraic} means that we require at least H\"older continuity with order greater than a half in order to guarantee uniform convergence.

\begin{theorem}\label{thm:subalgebraic1}
  If $f \in C^\alpha([-1,1])$, where $\alpha > \frac12$, then
  \begin{equation*}
  \|f - f_N\|_{L^\infty(-1,1)} \to 0 \text{ as } N \to \infty.
  \end{equation*}
\end{theorem}

In order to guarantee local, pointwise convergence, there is a weak local continuity condition which can be employed as follows. A function $f$ is \emph{uniformly Dini--Lipschitz} in $[a,b]$ if \cite{zygmund2002trigonometric},
\begin{equation}\label{eqn:DiniLipschitz}
\lim_{\delta \searrow 0}\sup_{\substack{x,y \in [a,b] \\ |x-y|<\delta}} \left|(f(x)- f(y)) \log \delta \right| = 0.
\end{equation}
This is a very weak condition, weaker than the H\"older condition for any $\alpha > 0$, but it is sufficient for convergence of Fourier extensions in the interior of $[-1,1]$.

\begin{theorem}\label{thm:DiniLipschitz}
  If $f \in L^2(-1,1)$ is uniformly Dini--Lipschitz in $[a,b] \subseteq [-1,1]$, then 
  \begin{equation*}
  \|f - f_N\|_{L^\infty(c,d)} \to 0 \text{ as } N \to \infty,
  \end{equation*}
  for all $[c,d] \subset (a,b)$.
 \end{theorem}

\begin{remark}
  This theorem is stronger than it might appear at first. It says that even if a function is in $L^2(-1,1)$, and can have for example jump discontinuities, we will still have pointwise convergence in regions where $f$ is Dini--Lipschitz. However, the localization theorem (Theorem \ref{thm:localization}) which we use to prove this result, does not give any indication of the \emph{rate} of convergence.
\end{remark}

\subsection{Proofs of the results of this section}

For each odd positive integer $N = 2n+1$, let $P_N$ be the orthogonal projection from $L^2(-1,1)$ onto the subspace $\Hspace_N$,
\begin{equation*}
\Hspace_N = \mathrm{span}\{e^{\frac{i\pi}{\T}kx}\}_{k=-n}^n.
\end{equation*}
Then $f_N = P_N(f)$, since $f_N$ minimizes the $L^2(-1,1)$ distance between $f$ and $\Hspace_N$. Let $\{ e_k \}_{k = 1}^N$ be \emph{any} orthonormal basis for $\Hspace_N \subset L^2(-1,1)$. Then the kernel, 
\begin{equation*}
K_N(x,y) = \sum_{k=1}^N e_k(x) \overline{e_k(y)},
\end{equation*}
satisfies
\begin{equation*}
\PN f(x) = \int_{-1}^1 K_N(x,y) f(y) \d y,
\end{equation*}
for all $f \in L^2(-1,1)$. The \emph{Lebesgue function} for the projection $\PN$ at a point $x\in[-1,1]$ is the $L^1$ norm of the kernel at $x$,
\begin{equation*}
\Lambda(x;\PN) = \int_{-1}^{1} \left|K_N(x,y) \right| \d y.
\end{equation*}
The \emph{best approximation error functional} on $\Hspace_N$ is defined for all $f\in C([-1,1])$ by
\begin{equation}\label{eqn:bestapproximation}
E(f;\Hspace_N) = \inf_{r_N \in \Hspace_N} \|f - r_N\|_{L^\infty(-1,1)}.
\end{equation}

The importance of $\Lambda(x;\PN)$ and $E(f;\Hspace_N)$ are encapsulated in Lebesgue's Lemma, which states that for any ${f\in C([-1,1])}$,
\begin{equation}\label{eqn:LebesgueLemma}
|f(x) - \PN(f)(x)| \leq (1+\Lambda(x;\PN))E(f;\Hspace_N),
\end{equation} 
for all $x \in [-1,1]$ \cite[Ch.~2, Prop.~4.1]{devore1993constructive}, \cite[Thm.~2.5.2]{phillips2003interpolation}.

Now we can proceed to prove the pointwise convergence results stated above. The proofs depend on the content of Sections \ref{sec:prolatekernel}, \ref{sec:bestuniform} and \ref{sec:localization}, which consist of self-contained results.

\begin{lemma}\label{lem:superlemma}
  Let $f \in C([-1,1])$. Then for all closed subsets $[a,b] \subset (-1,1)$, we have
  \begin{equation*}
  \|f - \PN(f)\|_{L^\infty(a,b)}= \mathcal{O}(\log N) E\left(f;\Hspace_N\right),
  \end{equation*}
  where the constant in the big $\mathcal{O}$ depends on $a$, $b$ and $T$. Over the whole interval $[-1,1]$, we have
    \begin{equation*}
    \|f - \PN(f)\|_{L^\infty(-1,1)} = \mathcal{O}(N^{\frac12}) E\left(f;\Hspace_N\right),
    \end{equation*}
  where the constant in the big $\mathcal{O}$ depends only on $T$.
  \begin{proof}
    By Lebesgue's Lemma, given in equation \eqref{eqn:LebesgueLemma}, it suffices to show that $\sup_{x \in [a,b]}\Lambda(x;\PN) = \mathcal{O}(\log N)$, and $\sup_{x \in [-1,1]}\Lambda(x; \PN) = \mathcal{O}(N^{\frac12})$. This is proved in Theorem \ref{thm:lebesguefunction}.
  \end{proof}
\end{lemma}

\begin{proof}[Proof of Theorem \ref{thm:Linftyalgebraic}]
  By Lemma \ref{lem:superlemma}, it suffices to show that for $f \in C^{k,\alpha}([-1,1])$, we have $E(f;\Hspace_N) = \mathcal{O}\left(N^{-k-\alpha}\right) |f|_{C^\alpha([-1,1])}$. This follows from Lemma \ref{lem:holdermodulus} and Theorem \ref{thm:Jacksonmain}.
\end{proof}

\begin{proof}[Proof of Theorem \ref{thm:subalgebraic1}] 
  This follows from Theorem \ref{thm:Linftyalgebraic} with $k = 0$, because $N^{\frac12 - \alpha} \log N \to 0$ as $N\to \infty$ for all $\alpha > \frac12$.
\end{proof}

\begin{proof}[Proof of Theorem \ref{thm:DiniLipschitz}]
  The following proof is an analogue of a proof of Freud for polynomial approximation (\cite[Thm.~IV.5.6]{freud1971orthogonal}). Define the functions $f_1$ and $f_2$ by
  \begin{equation*}
  f_1(x) = \begin{cases}
  f(x) & \text{ for } x \in [a,b] \\
  f(a) & \text{ for } x \in [-1,a) \\
  f(b) & \text{ for } x \in (b,1],
  \end{cases}
  \end{equation*}
  and $f_2 = f - f_1$. Since $f_2$ vanishes in $[a,b]$ and is in $L^2(-1,1)$, we have by Theorem \ref{thm:localization} that $\PN(f_2) \to 0$ uniformly in all subintervals $[c,d] \subset (a,b)$. It is clear by the definition of $f_1$ and the definition of Dini--Lipschitz continuity in equation \eqref{eqn:DiniLipschitz} that $f_1$ is also uniformly Dini--Lipschitz in $[-1,1]$. By Lemma \ref{lem:superlemma},
  \begin{equation*}
  \|f_1 - \PN(f_1)\|_{L^\infty(c,d)} = \mathcal{O}(\log N) E\left(f_1;\Hspace_N\right).
  \end{equation*}
  By Lemma \ref{lem:DiniLipschitzmodulus} and Theorem \ref{thm:Jacksonmain}, $E\left(f_1;\Hspace_N\right) = o(1/\log N)$. This proves that $\PN(f_1) \to f_1$ uniformly on all subintervals $[c,d] \subset (a,b)$. Now, since $f = f_1 + f_2$, we have proved the result.
\end{proof}

\section{The Lebesgue function of Fourier extensions}\label{sec:prolatekernel}

Recall from Section \ref{sec:pointwiseconvergence} that the kernel associated with the Fourier extension operator $P_N$ is the bivariate function on $[-1,1] \times [-1,1]$, 
\begin{equation*}
K_N(x,y) = \sum_{k=1}^N e_k(x) \overline{e_k(y)},
\end{equation*}
where $\{e_k\}_{k=1}^N$ is any orthonormal basis for $\Hspace_N$. We call this kernel the \emph{prolate kernel}, because one particular choice of orthonormal basis is the discrete prolate spheroidal wave functions (DPSWFs). These functions, denoted $\{\xi_{k,N} \}_{k=1}^N$, are the $N$ eigenfunctions of a time-band-limiting operator; specifically, there exist eigenvalues $\{\lambda_{k,N} \}_{k=1}^N$ such that
\begin{equation*}
 \int_{-1}^1 \xi_{k,N}(y) \frac{\sin\left(\frac{N \pi}{T}(x-y)\right)}{\sin\left(\frac{ \pi}{T}(x-y)\right)} \, \d y = \lambda_{k,N} \xi_{k,N}(x),
\end{equation*} 
for $k = 1,\ldots N$. DPSWFs play an important role in the analysis of perfectly bandlimited and nearly timelimited periodic signals, which was pioneered by Landau, Pollak and Slepian in the 1970s \cite{slepian1978prolate}. More recently, they have also been shown to be important for the computation of Fourier extensions, because the regularized version of Fourier extensions projects onto the DPSWFs $\xi_{k,N}$ with eigenvalues $\lambda_{k,N} > \varepsilon$ for a given tolerance $\varepsilon > 0$ \cite{adcock2014numerical,adcock2018frames}. This is discussed in Section \ref{sec:discussion}.

The key outcome of this section is a proof of the following theorem.

\begin{theorem}[Lebesgue function bounds]\label{thm:lebesguefunction}
\begin{enumerate}[(i)]
 \item For each closed interval $[a,b] \subset (-1,1)$, the Lebesgue function satisfies
 \begin{equation*}
 \sup_{x \in [a,b]}\Lambda(x;\PN) = \mathcal{O}(\log N).
 \end{equation*}
 \item Over the whole interval, $[-1,1]$, we have
   \begin{equation*}
   \sup_{x \in [-1,1]}\Lambda(x;\PN) = \mathcal{O}(N^{\frac12}).
   \end{equation*}
\end{enumerate}
\end{theorem}

This will be proved by finding asymptotic formulae for the prolate kernel $K_N$. The reader can verify that $K_N$ is invariant under a change of orthonormal basis for $\Hspace_N$, so a suitable choice of basis is desired. We have found that rather than the DPSWF basis, a basis related to orthogonal polynomials on the unit circle have been more amenable to analysis. For $N = 2n+1$, recall the definition of the $N$-dimensional space $\Hspace_N$,
\begin{equation*}
\Hspace_N = \myspan\left\{e^{\frac{i\pi}{\T}kx}\right\}_{k=-n}^n.
\end{equation*}
Any function $r_N \in \mathcal{H}_N$ is of the form
\begin{equation*}
r_N(x) = e^{-\frac{i\pi}{\T}nx} p_{2n}(e^{\frac{i\pi}{\T}x}),
\end{equation*}
where $p_{2n}$ is a polynomial of degree $2n$. Using this idea we prove the following lemma.
\begin{lemma}[Orthonormal basis for $\Hspace_N$]\label{lem:HNonb}
  Let $\left\{\Pi_{k}(z)\right\}_{k=0}^{\infty}$ be the (normalized) orthogonal polynomials on the unit circle with respect to the weight
  \begin{equation*}
   f(\theta) = 2\T \cdot \chi_{\left[-\frac{\pi}{\T},\frac{\pi}{\T}\right]}(\theta), \quad \theta \in [-\pi,\pi],
  \end{equation*}
    i.e. for $j,k = 0,1,2,\ldots$,
    \begin{equation*}
    \frac{1}{2\pi}\int_{-\pi}^\pi \overline{\Pi_k(e^{i\theta})}\Pi_j(e^{i\theta}) \, f(\theta)\d\theta = \delta_{j,k}.
    \end{equation*}
   Then the set
    \begin{equation*}
    \left\{ e^{-\frac{i\pi}{\T}nx} \cdot \Pi_k\left( e^{\frac{i\pi}{\T}x} \right) \right\}_{k=0}^{2n}
    \end{equation*}
    forms an orthonormal basis for $\mathcal{H}_N$.
    \begin{proof}
      By the observation immediately preceding this lemma, the set forms a basis for $\mathcal{H}_N$ because $\{\Pi_k\}_{k=0}^{2n}$ forms a basis for polynomials of degree $2n$. We need only show its orthonormality with respect to the inner product on $\Hspace_N$ induced by $L^2(-1,1)$.
      Let $j , k \in \{0,\ldots, 2n\}$. Then, making the change of variables $\theta = \frac{\pi}{\T}x$, we have
      \begin{eqnarray*}
        \int_{-1}^1 \overline{e^{-\frac{i\pi}{\T}nx} \cdot \Pi_j\left( e^{\frac{i\pi}{\T}x} \right)} e^{-\frac{i\pi}{\T}nx} \cdot \Pi_k\left( e^{\frac{i\pi}{\T}x} \right) \, \d x &=& \int_{-1}^1 \overline{\Pi_j\left( e^{\frac{i\pi}{\T}x} \right)} \Pi_k\left( e^{\frac{i\pi}{\T}x} \right) \, \d x \\
         &=& \int_{-\frac{\pi}{\T}}^{\frac{\pi}{\T}} \overline{\Pi_j\left(e^{i\theta} \right)} \Pi_k\left(e^{i\theta}\right) \frac{\T}{\pi}\,\d \theta \\
         &=& \frac{1}{2\pi}\int_{-\pi}^\pi \overline{\Pi_j(e^{i\theta})}\Pi_k(e^{i\theta}) \, f(\theta)\d\theta. 
      \end{eqnarray*}
    By the orthonormal relationship between $\Pi_k$ and $\Pi_j$ on the unit circle, the basis is orthonormal on $[-1,1]$.
    \end{proof}
\end{lemma}

The Christoffel-Darboux formula for orthogonal polynomials on the unit circle states that,
\begin{equation*}
\sum_{k=0}^{N-1} \overline{\Pi_k(\zeta)} \Pi_k(z)  = \frac{\overline{\Pi^*_N(\zeta)} \Pi^*_N(z) - \overline{\Pi_N(\zeta)} \Pi_N(z) }{1-\overline{\zeta}z}, \qquad z, \zeta \in \mathbb{C}, \quad \overline{\zeta}z \neq 1,
\end{equation*}
where $\Pi^*_N(z) = z^N \overline{\Pi\left(\overline{z}^{-1}\right)}$ (which is also a polynomial of degree $N$) \cite[Thm 11.42]{szeg1939orthogonal}. On the unit circle itself, where $z = e^{i\theta}$, $\zeta=e^{i\phi}$, this reduces, after some elementary manipulations, to
\begin{equation}\label{eqn:ChristoffelDarboux}
\sum_{k=0}^{N-1} \overline{\Pi_k(e^{i\phi})} \Pi_k(e^{i\theta}) = e^{i\frac{N-1}{2}(\theta - \phi)} \cdot \mathrm{Imag}\left( \frac{\overline{e^{-i\frac{N}{2}\phi} \cdot \Pi_N\left(e^{i\phi}\right)} \cdot e^{-i\frac{N}{2}\theta} \cdot \Pi_N\left(e^{i\theta}\right)   }{\sin\left(\frac{\theta-\phi}{2} \right)}\right).
\end{equation}
From this general formula for orthogonal polynomials on the unit circle, we prove the following lemma regarding the prolate kernel.

\begin{lemma}[Prolate kernel formula]\label{lem:prolatekernelformula}
  For all $x,y \in \left[ -1,1 \right]$,
  \begin{equation*}
  K_N(x,y) = \mathrm{Imag}\left( \frac{\overline{e^{-\frac{i\pi}{\T}\frac{N}{2}y} \cdot \Pi_N\left(e^{\frac{i\pi}{T}y}\right)} \cdot e^{-\frac{i\pi}{\T}\frac{N}{2}x} \cdot \Pi_N\left(e^{\frac{i\pi}{T}x}\right)   }{\sin\left(\frac{\pi}{2\T}(x-y) \right)}\right).
  \end{equation*}
  The formula in fact holds for all $x,y \in [-T,T]$. 
  \begin{remark}
    Setting $\T = 1$ in this formula returns the Dirichlet kernel of classical Fourier series, because $\Pi_N(z) = z^N$ for the trivial weight $f(\theta) \equiv 1$.
  \end{remark}
  \begin{proof}
    By the fact that $\left\{ e^{-\frac{i\pi}{\T}nx} \cdot \Pi_k\left( e^{\frac{i\pi}{\T}x} \right) \right\}_{k=0}^{2n}$ is an orthonormal basis for $\Hspace_N$, from Lemma \ref{lem:HNonb}, we have that
    \begin{eqnarray*}
      K_N(x,y) &=& \sum_{k=0}^{2n} \overline{e^{-\frac{i\pi}{\T}ny} \Pi_k\left( e^{\frac{i\pi}{\T}y}\right)} e^{-\frac{i\pi}{\T}nx} \Pi_k\left( e^{\frac{i\pi}{\T}x}\right)  \\
      &=& e^{\frac{i\pi}{\T}n(y-x)} \sum_{k=0}^{N-1} \overline{\Pi_k\left(e^{\frac{i\pi}{\T}y}\right)}  \Pi_k\left(e^{\frac{i\pi}{\T}x}\right).
    \end{eqnarray*}
  The proof is completed by considering the Christoffel-Darboux formula for orthogonal polynomials on the unit circle in equation \eqref{eqn:ChristoffelDarboux} (note that $\frac{N-1}{2} = n$).  
  \end{proof}
\end{lemma}

Now, to ascertain asymptotics of the prolate kernel, it is sufficient to ascertain asymptotics of the orthogonal polynomials $\{\Pi_k(z)\}_{k=0}^\infty$. These polynomials have been studied before in the literature, and are known as the Legendre polynomials on a circular arc \cite{magnus2006freud}.
\begin{theorem}\label{thm:asymptotics}
  Let $\{\Pi_k\}_{k=0}^\infty$ be the (normalized) orthogonal polynomials on the unit circle with respect to the weight $f(\theta) = 2\T \cdot \chi_{\left[-\frac{\pi}{\T},\frac{\pi}{\T}\right]}(\theta)$, and for $x\in[-1,1]$ define the variable $\eta \in [0,\pi]$ by
  \begin{equation*}
  \eta = \cos^{-1}\left(\frac{\sin\left(x\frac{\pi}{2\T} \right)}{\sin\left(\frac{\pi}{2\T} \right)}\right).
  \end{equation*} There exists a constant $\delta > 0$ such that for $x \in [-1+\delta,1-\delta]$, 
  \begin{eqnarray}
   \Pi_N\left(e^{\frac{i\pi}{\T}x}\right) &=& \frac{e^{\frac{i\pi}{\T}\frac{N}{2}x}}{\sqrt{2\T\sin\left(\frac{\pi}{2\T}\right)}} \bigg(e^{-\frac{i\pi}{4\T}}\left(\frac{\sin\left((1+x)\frac{\pi}{2\T}\right)}{\sin\left((1-x)\frac{\pi}{2\T}\right)}\right)^{\frac14}\cos\left(N\eta - \frac{\pi}{4}\right) \\ 
  & & \qquad\qquad\qquad - \quad e^{\frac{i\pi}{4\T}}\left(\frac{\sin\left((1-x)\frac{\pi}{2\T}\right)}{\sin\left((1+x)\frac{\pi}{2\T}\right)}\right)^{\frac14}\sin\left(N\eta - \frac{\pi}{4}\right)\bigg) + \mathcal{O}(N^{-1}), \nonumber
  \end{eqnarray}
  and for $x \in [1-\delta,1]$,
  \begin{eqnarray}
    \Pi_N\left(e^{\frac{i\pi}{\T}x}\right) &=& \frac{  e^{\frac{i\pi}{\T}\frac{N}{2}x}}{\sqrt{2\T\sin\left(\frac{\pi}{2\T}\right)}}\left(\frac{\pi}{2}N\eta\right)^{\frac12} \bigg(e^{-\frac{i\pi}{4\T}}\left(\frac{\sin\left((1+x)\frac{\pi}{2\T}\right)}{\sin\left((1-x)\frac{\pi}{2\T}\right)}\right)^{\frac14}J_0\left(N\eta\right) \\
  & & \qquad\qquad\qquad\qquad\qquad - \quad e^{\frac{i\pi}{4\T}}\left(\frac{\sin\left((1-x)\frac{\pi}{2\T}\right)}{\sin\left((1+x)\frac{\pi}{2\T}\right)}\right)^{\frac14}J_1\left(N\eta\right)\bigg) + \mathcal{O}(N^{-\frac12}). \nonumber
  \end{eqnarray}
  The constants in the big $\mathcal{O}$ depend only on $T$ and $\delta$. The asymptotics for $x\in[-1,-1+\delta]$ are found by using the relation $\Pi_N\left(e^{-\frac{i\pi}{\T}x}\right) = \overline{\Pi_N\left(e^{\frac{i\pi}{\T}x}\right)}$. 
  
  In terms of magnitude with respect to $N$, we have
  \begin{equation}\label{eqn:asymptoticsizePiN}
  \Pi_N\left(e^{\frac{i\pi}{\T}x}\right) = \begin{cases} \mathcal{O}(1) & \text{ for } x \in [-1+\delta,1-\delta], \\
                                                          \mathcal{O}(N^{\frac12}) & \text{ for } x \in [-1,-1+\delta] \cup [1-\delta,1]. \end{cases}
  \end{equation}
  \begin{remark}
     The asymptotic order of $\Pi_N\left(e^{\frac{i\pi}{\T}x}\right)$ with respect to $N$ in equation \eqref{eqn:asymptoticsizePiN} is the same as for the $N$th (normalized) Legendre polynomial in $[-1,1]$ \cite[Thm.~8.21.6]{szeg1939orthogonal}. Further discussion on how Legendre series approximations compare to Fourier extensions lies in subsection 8.1.
  \end{remark}
  \begin{proof}
    This result follows directly from Lemma \ref{lem:asymArcLeg}, because if we take $\alpha = \pi - \pi/\T$ and $f_\alpha(\theta) \equiv 1$, then the polynomials $\Pi_N(z) = (2\T)^{-\frac12}\phi_N\left(-z,\alpha\right)$ satisfy the orthonormality conditions that define $\Pi_N$ as in Lemma \ref{lem:HNonb}. To obtain the asymptotic formula above, make the change of variables $\theta = \frac{\pi}{\T}x + \pi$ in the asymptotic formulae for $\phi_N(z,\alpha)$. Be careful to note that the endpoint with explicit formula given above ($x = 1$), corresponds to $\theta = 2\pi - \alpha$, which is not the endpoint with explicit formula given in Lemma \ref{lem:asymArcLeg} ($\theta = \alpha$). This was done to shorten the expressions for the asymptotics at the endpoints.

    To complete the proof we must prove equation \eqref{eqn:asymptoticsizePiN}. For $x \in [-1+\delta,1-\delta]$, all of the terms are clearly bounded above by $\mathcal{O}(\delta^{-\frac14}) = \mathcal{O}(1)$. Now let $x\in[1-\delta,1]$. We have $\eta^2 \leq \frac{\pi^2}{4}(1-\cos(\eta))$ for all $\eta \in \left[0,\frac{\pi}{2}\right]$ and $x \sin\left(\frac{\pi}{2\T}\right)  \leq \sin\left(x\frac{\pi}{2\T}\right) \leq x\frac{\pi}{2\T}$ for all $x \in [0,1]$. Assuming $\delta < \frac12$, we have $x,y \in [0,1]$ and $\eta,\lambda \in \left[0,\frac{\pi}{2}\right]$, and hence $\eta^2 \leq \frac{\pi^2}{4}(1-x)$. Since $1-x \in [0,1]$ we have $1-x \leq \sin\left((1-x)\frac{\pi}{2\T}\right)/\sin\left(\frac{\pi}{2\T}\right)$, so $\eta^2 \leq \sin\left((1-x)\frac{\pi}{2\T}\right) \frac{\pi^2}{4\sin\left(\frac{\pi}{2\T}\right)}$. This implies that
    \begin{equation*}
    \left(\frac{\eta^2\sin\left((1+x)\frac{\pi}{2\T}\right)}{\sin\left((1-x)\frac{\pi}{2\T}\right)}\right)^{\frac14} = \mathcal{O}(1),
    \end{equation*}
    uniformly for all $x \in [0,1]$. Note also that Bessel functions are uniformly bounded in absolute value by $1$ (see \cite[Eq.~10.14.1]{NIST:DLMF}). This makes it clear that $\Pi_N\left(e^{\frac{i\pi}{\T}x}\right) = \mathcal{O}(N^{\frac12})$ for $x\in[1-\delta,1]$. For $x \in [-1,-1+\delta]$, use the relation $\Pi_N\left(e^{-\frac{i\pi}{\T}x}\right) = \overline{\Pi_N\left(e^{\frac{i\pi}{\T}x}\right)}$.
  \end{proof}
\end{theorem}

We now have the required results to prove Theorem \ref{thm:lebesguefunction}.

\begin{proof}[Proof of Theorem \ref{thm:lebesguefunction} part (i)]
Let $[a,b] \subset (-1,1)$ and choose $\tau > 0$ sufficiently small so that $[a-\tau,b+\tau] \subset (-1,1)$. Applying the first part of Theorem \ref{thm:asymptotics} gives us
\begin{equation}\label{eqn:o1bound}
\Pi_N\left(e^{\frac{i\pi}{\T}x}\right) = \mathcal{O}(1), \quad x \in [a-\tau,b+\tau].
\end{equation}
For the proof of part (i) we need to bound the integral $\int_{-1}^1 |K_N(x,y)| \,\d y$ uniformly for $x \in [a,b]$. We do so by dividing the interval $[-1,1]$ into the following subsets:
\begin{eqnarray*}
I_1 &=& \left\{y \in [-1,1] : |y - x| \leq N^{-1} \right\} \\
I_2 &=& \left\{y \in [-1,1] : N^{-1} < |y - x| \leq \tau \right\}\\
I_3 &=& \left\{y \in [-1,1] : \tau < |y - x| \right\}.
\end{eqnarray*}
We will obtain estimates for the kernel for $x \in [a,b]$ and $y$ in each of $I_1$, $I_2$ and $I_3$, and then estimate the associated integral over each of $I_1$, $I_2$ and $I_3$.

For $N > 1 / \tau$, we have $I_1$ and $I_2$ are nonempty and are contained within $[a - \tau, b+ \tau] \subset (-1,1)$. By equation \eqref{eqn:o1bound}, 
\begin{equation*}
\Pi_N\left(e^{\frac{i\pi}{\T}y}\right) = \mathcal{O}(1), \quad y \in I_1 \cup I_2.
\end{equation*}
 For $y \in I_1$, we have
\begin{equation*}
K_N(x,y) = e^{\frac{i\pi}{\T}n(y-x)} \sum_{k=0}^{N-1} \overline{\Pi_k(e^{\frac{i\pi}{\T}y})}  \Pi_k(e^{\frac{i\pi}{\T}x}) = \mathcal{O}(N).
\end{equation*}
This implies
\begin{equation*}
\int_{I_1} |K_N(x,y)| \,\d y \leq \mathcal{O}(N) \int_{I_1} \,\d y = \mathcal{O}(1),
\end{equation*}
because $|I_1| \leq 2N^{-1}$. 

By Lemma \ref{lem:prolatekernelformula},
\begin{equation*}
K_N(x,y) = \mathrm{Imag}\left( \frac{\overline{e^{-\frac{i\pi}{\T}\frac{N}{2}y} \cdot \Pi_N\left(e^{\frac{i\pi}{T}y}\right)} \cdot e^{-\frac{i\pi}{\T}\frac{N}{2}x} \cdot \Pi_N\left(e^{\frac{i\pi}{T}x}\right)   }{\sin\left(\frac{\pi}{2\T}(x-y) \right)}\right).
\end{equation*}
Note that since the sine function is concave in $[0,\pi]$, we have $|\sin\left(\frac{\pi}{2\T}(x-y)\right)| \geq \sin\left(\frac{\pi}{2\T}\right) |x-y|$ for $x,y \in [-1,1]$. Therefore, for all $y \in [-1,1]$,
\begin{equation*}
|K_N(x,y)| \leq \mathcal{O}(1)\frac{1}{|x-y|} \left|\Pi_N\left(e^{\frac{i\pi}{T}y}\right)\right|
\end{equation*}
For $y \in I_2$, this can be reduced to $|K_N(x,y)| \leq \mathcal{O}(1)\frac{1}{|x-y|}$. Therefore,
\begin{equation*}
\int_{I_2} |K_N(x,y)| \, \d y \leq \mathcal{O}(1) \int_{I_2} \frac{1}{|x-y|} \, \d y \leq \mathcal{O}(1) \int_{N^{-1}}^{\tau} \frac{1}{s} \, \d s = \mathcal{O}(\log(N)).
\end{equation*}
For $y \in I_3$, since $|x-y|^{-1} < \tau^{-1} = \mathcal{O}(1)$, we have $|K_N(x,y)| \leq \mathcal{O}(1) \left|\Pi_N\left(e^{\frac{i\pi}{T}y}\right)\right|$. Therefore,
\begin{equation*}
\int_{I_3} |K_N(x,y)| \, \d y \leq \mathcal{O}(1) \int_{I_3} \left|\Pi_N\left(e^{\frac{i\pi}{T}y}\right)\right| \, \d y \leq \mathcal{O}(1) \left(\int_{-1}^1  \left|\Pi_N\left(e^{\frac{i\pi}{T}y}\right)\right|^2 \, \d y \right)^{\frac12} = \mathcal{O}(1).
\end{equation*}
This proves that $\Lambda(x;\PN) = \mathcal{O}(\log(N))$ uniformly for $x \in [a,b]$.
\end{proof}

\begin{proof}[Proof of Theorem \ref{thm:lebesguefunction} part (ii)]
  Let $\delta \in \left(0,\frac14\right)$ be sufficiently small so that Theorem \ref{thm:asymptotics} applies to the intervals $[-1+2\delta,1-2\delta]$ and $[1-2\delta,1]$. Using part (i) of the present theorem, we have that for all $x \in [-1+\delta, 1-\delta]$ the Lebesgue function satisfies $\Lambda(x;\PN) = \mathcal{O}(\log(N)) = \mathcal{O}(N^{\frac12})$ uniformly in such $x$. Now, since $\Pi_N\left(e^{-\frac{i\pi}{\T}x}\right) = \overline{\Pi_N\left(e^{\frac{i\pi}{\T}x}\right)}$, it follows that $K_N(-x,y) = \overline{K_N(x,-y)}$, so that $\Lambda(-x;\PN) = \Lambda(x;\PN)$. Therefore, to complete the proof we need only show that $\Lambda(x;\PN) = \mathcal{O}(N^{\frac12})$ uniformly for $x \in [1-\delta,1]$. For such $x$, we divide the interval $[-1,1]$ into the following subsets:
  \begin{eqnarray*}
  I_1 &=& \left\{y \in [-1,1] : |y - x| \leq N^{-1} \text{ or } |1-y| \leq N^{-1}\right\} \\
  I_2 &=& \left\{y \in [-1,1] : N^{-1} < |y - x| \leq \delta  \text{ and } |1-y| > N^{-1}\right\}\\
  I_3 &=& \left\{y \in [-1,1] : \delta < |y - x| \text{ and } |1-y| > N^{-1} \right\}.
  \end{eqnarray*}
  By Theorem \ref{thm:asymptotics}, 
  \begin{equation*}
  \Pi_N\left(e^{\frac{i\pi}{\T}x}\right) = \mathcal{O}(N^{\frac12}), \quad x \in [-1,1].
  \end{equation*}
  Therefore,
  \begin{equation*}
  K_N(x,y) = e^{\frac{i\pi}{\T}n(y-x)} \sum_{k=0}^{N-1} \overline{\Pi_k(e^{\frac{i\pi}{\T}y})}  \Pi_k(e^{\frac{i\pi}{\T}x}) = \mathcal{O}(N^2).
  \end{equation*}
  By the Cauchy-Schwarz inequality and the fact that $|I_1| \leq 3N^{-1}$, we have
  \begin{equation*}
  \int_{I_1} |K_N(x,y)| \,\d y \leq \left(\int_{I_1} |K_N(x,y)|^2 \,\d y \right)^{\frac12} \left(\int_{I_1} \,\d y\right)^{\frac12} \leq \left(\frac{3}{N}\right)^{\frac12}\left(\int_{-1}^1 |K_N(x,y)|^2 \,\d y \right)^{\frac12}.
  \end{equation*}
  By the connection between $K_N$ and $P_N$, $\int_{-1}^1 |K_N(x,y)|^2 \,\d y = P_N\left(\overline{K_N(x,\cdot)}\right)(x)$. Since $\overline{K_N(x,y)} = K_N(y,x)$ and because $K_N(\cdot,x) \in \Hspace_N$ for each $x \in [-1,1]$, we have,
  \begin{equation*}
  \int_{-1}^1 |K_N(x,y)|^2 \,\d y = K_N(x,x).
  \end{equation*}
 Therefore,
  \begin{equation*}
  \int_{I_1} |K_N(x,y)| \,\d y = \mathcal{O}(N^{-\frac12}) \left(\mathcal{O}(N^2)\right)^{\frac12} = \mathcal{O}(N^{\frac12}).
  \end{equation*}
Just as in the proof of part (i) of the theorem, but this time using the estimate $\Pi_N\left(e^{\frac{i\pi}{\T}x}\right) = \mathcal{O}(N^{\frac12})$, we have for all $x,y \in [-1,1]$,
  \begin{equation*}
  |K_N(x,y)| \leq \mathcal{O}(N^{\frac12})\frac{1}{|x-y|} \left|\Pi_N\left(e^{\frac{i\pi}{T}y}\right)\right|.
  \end{equation*}
  Therefore, for $y \in I_3$,
  \begin{equation*}
  |K_N(x,y)| \leq \mathcal{O}(N^{\frac12})\left|\Pi_N\left(e^{\frac{i\pi}{T}y}\right)\right|,
  \end{equation*}
  because $|x-y| > \delta$ for $y\in I_3$. Hence, 
  \begin{equation*}
  \int_{I_3} |K_N(x,y)| \, \d y \leq \mathcal{O}(N^{\frac12}) \int_{I_3} \left|\Pi_N\left(e^{\frac{i\pi}{T}y}\right)\right| \, \d y \leq \mathcal{O}(N^{\frac12}) \left(\int_{-1}^1  \left|\Pi_N\left(e^{\frac{i\pi}{T}y}\right)\right|^2 \, \d y \right)^{\frac12} = \mathcal{O}(N^{\frac12}).
  \end{equation*}
  
  All that remains is to show that $\int_{I_2} |K_N(x,y)| \, \d y = \mathcal{O}(N^{\frac12})$ uniformly for $x \in [1-\delta,1]$. For $x \in [1-\delta,1]$ and $y \in I_2$, we have $y \in [1-2\delta,1]$ so that the asymptotic expression in Theorem \ref{thm:asymptotics} holds. Define the variables
  \begin{equation*}
  \eta = \cos^{-1}\left(\frac{\sin\left(x\frac{\pi}{2\T} \right)}{\sin\left(\frac{\pi}{2\T} \right)}\right), \qquad   \lambda = \cos^{-1}\left(\frac{\sin\left(y\frac{\pi}{2\T} \right)}{\sin\left(\frac{\pi}{2\T} \right)}\right).
  \end{equation*}
  
  Take the asymptotic expressions for $\Pi_N$ in Theorem \ref{thm:asymptotics} for the $x$ and $y$ currently in question, and consider the numerator in the formula for the kernel $K_N(x,y)$ (Lemma \ref{lem:prolatekernelformula}). An asymptotic formula is as follows:
  \begin{eqnarray}
  & & \mathrm{Imag}\left(\overline{e^{-\frac{i\pi}{\T}\frac{N}{2}y} \cdot \Pi_N\left(e^{\frac{i\pi}{T}y}\right)} \cdot e^{-\frac{i\pi}{\T}\frac{N}{2}x} \cdot \Pi_N\left(e^{\frac{i\pi}{T}x}\right) \right) \\
  \nonumber& & \qquad = \frac{1}{2\T} \left(\frac{\pi}{2}N\eta\right)^{\frac12} \left(\frac{\pi}{2}N\lambda\right)^{\frac12}\\
  & & \qquad \qquad \cdot\Bigg( \left(\frac{\sin\left((1+x)\frac{\pi}{2\T}\right)}{\sin\left((1-x)\frac{\pi}{2\T}\right)}\right)^{\frac14} J_0(N\eta) \left(\frac{\sin\left((1-y)\frac{\pi}{2\T}\right)}{\sin\left((1+y)\frac{\pi}{2\T}\right)}\right)^{\frac14}J_1\left(N\lambda\right) \\
  \nonumber& & \qquad \qquad\qquad - \left(\frac{\sin\left((1-x)\frac{\pi}{2\T}\right)}{\sin\left((1+x)\frac{\pi}{2\T}\right)}\right)^{\frac14} J_1(N\eta) \left(\frac{\sin\left((1+y)\frac{\pi}{2\T}\right)}{\sin\left((1-y)\frac{\pi}{2\T}\right)}\right)^{\frac14} J_0\left(N\lambda\right)  \Bigg)  + \mathcal{O}(1) \\
 \nonumber & & \qquad = N^{\frac12}\frac{\pi}{4\T} \\
  & & \qquad \cdot\Bigg(\left(\frac{\eta^2\sin\left((1+x)\frac{\pi}{2\T}\right)}{\sin\left((1-x)\frac{\pi}{2\T}\right)}\right)^{\frac14} J_0(N\eta) \left(\frac{\sin\left((1-y)\frac{\pi}{2\T}\right)}{\sin\left((1+y)\frac{\pi}{2\T}\right)}\right)^{\frac14}  \left(N\lambda\right)^{\frac12}J_1\left(N\lambda\right) \\
 \nonumber & & \qquad \quad - \left(\frac{\sin\left((1-x)\frac{\pi}{2\T}\right)}{\sin\left((1+x)\frac{\pi}{2\T}\right)}\right)^{\frac14} \left(N\eta\right)^{\frac12}J_1\left(N\eta\right)\left(\frac{\lambda^2\sin\left((1+y)\frac{\pi}{2\T}\right)}{\sin\left((1-y)\frac{\pi}{2\T}\right)}\right)^{\frac14}  J_0(N\lambda) \Bigg) + \mathcal{O}(1).
  \end{eqnarray}
  This was an important step in the proof, because there was cancellation when we took the imaginary part. This cancellation is essential for the result to hold, and it is the reason for deriving and including a fully explicit description of the leading order asymptotics of the polynomials in Appendix \ref{A:legendre}.
  
 We will now proceed to find upper bounds on the resulting expression. We showed in the proof of Theorem \ref{thm:asymptotics} that,
  \begin{equation*}
  \left(\frac{\eta^2\sin\left((1+x)\frac{\pi}{2\T}\right)}{\sin\left((1-x)\frac{\pi}{2\T}\right)}\right)^{\frac14} = \mathcal{O}(1),
  \end{equation*}
  uniformly for all $x \in [0,1]$. The same is true when $x$ and $\eta$ are replaced by $y$ and $\lambda$.

  It is straightforward to also show that $(1-y)\frac{\pi}{2\T} \leq \sin\left(\frac{\pi}{2\T} \right)\lambda^2$ for $y \in [0,1]$ and $\lambda \in \left[0,\frac{\pi}{2}\right]$. From this, we have that for $y \in I_2$, $\lambda \geq \sqrt{\frac{\pi}{2\T N}}$. Combining this with the fact that for $t \to \infty$, $J_\alpha(t) = \mathcal{O}\left(t^{-\frac12}\right)$ (see \cite[Eq.~10.17.3]{NIST:DLMF}) we get that $J_0(N\lambda) = \mathcal{O}\left( N^{-\frac14}\right)$. 
  
  Note also that Bessel functions are uniformly bounded in absolute value by $1$ (see \cite[Eq.~10.14.1]{NIST:DLMF}). Furthermore, as $t \to \infty$, $t^{\frac12}J_\alpha(t) = \mathcal{O}(1)$ (see \cite[Eq.~10.17.3]{NIST:DLMF}). Collecting the bounds mentioned in the last three paragraphs, we conclude that for $y \in I_2$, we have,
  \begin{equation}\label{eqn:yI2bound}
  \mathrm{Imag}\left(\overline{e^{-\frac{i\pi}{\T}\frac{N}{2}y} \cdot \Pi_N\left(e^{\frac{i\pi}{T}y}\right)} \cdot e^{-\frac{i\pi}{\T}\frac{N}{2}x} \cdot \Pi_N\left(e^{\frac{i\pi}{T}x}\right) \right) = \mathcal{O}(N^{\frac12})J_0(N\eta)(1-y)^{\frac14} + \mathcal{O}(N^{\frac14}).
  \end{equation}
  
  To conclude, we prove two refinements of equation \eqref{eqn:yI2bound}, depending on whether $x \in [1-\delta,1-N^{-1}]$ or $x \in [1-N^{-1},1]$. When $x \in [1-\delta,1-N^{-1}]$, we have $J_0(N\eta) = \mathcal{O}(N^{-\frac14})$ (just like for $y \in I_2$ discussed above), and so,
  \begin{equation*}
  \mathrm{Imag}\left(\overline{e^{-\frac{i\pi}{\T}\frac{N}{2}y} \cdot \Pi_N\left(e^{\frac{i\pi}{T}y}\right)} \cdot e^{-\frac{i\pi}{\T}\frac{N}{2}x} \cdot \Pi_N\left(e^{\frac{i\pi}{T}x}\right) \right) =  \mathcal{O}(N^{\frac14}).
  \end{equation*}
  This implies that $K_N(x,y) = \mathcal{O}\left( \frac{N^{\frac14}}{|x-y|} \right)$ for $x \in [1-\delta,1-N^{-1}]$ and $y\in I_2$. Therefore, 
  \begin{equation*}
  \int_{I_2} |K_N(x,y)| \, \d y \leq \mathcal{O}(N^{\frac14}) \int_{I_2} \frac{1}{|x-y|} \, \d y \leq \mathcal{O}(N^{\frac14}) \int_{N^{-1}}^{\delta} \frac{1}{s} \, \d s = \mathcal{O}(N^{\frac14}\log(N)).
  \end{equation*}
  Finally, when $x \in [1- N^{-1},1]$ and $y \in I_2$, we have $1-y = x-y + 1-x \leq x - y + N^{-1}$ (since $x \geq y$). By concavity of the function $t \mapsto |t|^{\frac14}$ at $t = x-y > 0$, we have
  \begin{equation*}
  (x-y + N^{-1})^{\frac14} \leq (x-y)^{\frac14} + \frac14 N^{-1}(x-y)^{-\frac34}.
  \end{equation*}
Substituting this bound into equation \eqref{eqn:yI2bound}, we get,
  \begin{eqnarray*}
    K_N(x,y) &=& \mathcal{O}\left( \frac{N^{\frac12}(1-y)^{\frac14}}{|x-y|} \right)\\ 
          &=& \mathcal{O}\left(N^{\frac12}|x-y|^{-\frac34}\right) + \mathcal{O}\left(N^{-\frac12}|x-y|^{-\frac74}\right)
  \end{eqnarray*}
The integral is bounded in the predictable manner as follows:
  
  \begin{eqnarray*}
  \int_{I_2} |K_N(x,y)| \,\d y &=& \mathcal{O}(N^{\frac12})\int_{I_2} |x-y|^{-\frac34}\,\d y + \mathcal{O}(N^{-\frac12})\int_{I_2} |x-y|^{-\frac74} \,\d y \\
                         &\leq& \mathcal{O}(N^{\frac12}) + \mathcal{O}(N^{-\frac12})\int_{N^{-1}}^1 s^{-\frac74} \, \d s \\
                         &=&\mathcal{O}(N^{\frac12}) + \mathcal{O}(N^{-\frac12} \cdot N^{\frac34}) \\
                         &=& \mathcal{O}(N^{\frac12}).
  \end{eqnarray*}
  Since this covers all $x \in [-1,1]$ with finitely many uniform $\mathcal{O}(N^{\frac12})$ upper bounds, we have the final result uniformly for all $x \in [-1,1]$.
\end{proof}

\section{Best uniform approximation by Fourier extensions}\label{sec:bestuniform}

We will compare best uniform approximation in three spaces. For odd positive integer $N = 2n+1$, define,
\begin{eqnarray*}
\Hspace_N &=& \myspan\left\{e^{\frac{i\pi}{\T}kx}\right\}_{k=-n}^n \subset C([-1,1]), \\
\mathcal{T}_N &=& \myspan\left\{e^{\frac{i\pi}{\T}kx}\right\}_{k=-n}^n \subset C_{\per}([-T,T]), \\
\mathcal{P}_N &=& \myspan\left\{x^k\right\}_{k=0}^{2n} \subset C([-1,1]).
\end{eqnarray*}
We will see that best uniform approximation by Fourier extensions is more similar to that of algebraic polynomials. The best approximation error functionals for these spaces are defined by
\begin{eqnarray*}
E(f;\Hspace_N) &=& \inf_{r_N \in \Hspace_N}\|f-r_N\|_{L^\infty(-1,1)} \quad\text{ for all } f \in C([-1,1]), \\
E(g;\mathcal{T}_N) &=& \inf_{t_N \in \mathcal{T}_N}\|g-t_N\|_{L^\infty(-T,T)} \quad\text{ for all } g \in C_{\per}([-T,T]), \\
E(h;\mathcal{P}_N) &=& \inf_{p_N \in \mathcal{P}_N}\|h-p_N\|_{L^\infty(-1,1)} \quad \text{ for all } h \in C([-1,1]).
\end{eqnarray*}
We wish to find bounds in terms of $N$ and the regularity of the functions to be approximated.

For $f \in C([-1,1])$ the \emph{modulus of continuity} is defined by \cite{devore1993constructive,mhaskar2000fundamentals},
\begin{equation}\label{eqn:modulusofcty}
\omega(f;\delta) = \sup_{\substack{x,y \in[-1,1]\\|x-y|<\delta}} |f(x) - f(y)|.
\end{equation}
For $g \in C_{\per}([-T,T])$ we define the periodic modulus of continuity to be,
\begin{equation*}
\omega_{\per}(g;\delta) = \sup_{\substack{x,y \in [-T,T]\\d_T(x,y)<\delta}} |g(x) - g(y)|,
\end{equation*}
where $d_T(x,y)$ is the distance between $x,y$ as elements of the periodic interval $[-T,T]$. The following results are immediate.
\begin{lemma}\label{lem:holdermodulus}
  If $f$ is in the Holder space $C^\alpha([-1,1])$ for $\alpha \in [0,1]$, then $\omega(f;\delta) \leq \delta^\alpha |f|_{C^\alpha([-1,1])}$ for all $\delta > 0$. 
\end{lemma}
\begin{lemma}\label{lem:DiniLipschitzmodulus}
  If $f \in C([-1,1])$ is uniformly Dini--Lipschitz \cite{zygmund2002trigonometric}, i.e.
  \begin{equation*}
  \lim_{\delta \searrow 0}\sup_{\substack{x,y \in [-1,1]\\ |x-y|<\delta}} \left|(f(x)- f(y)) \log \delta \right| = 0,
  \end{equation*}
  then $\omega(f;\delta) = o(1 / |\log \delta|)$.
\end{lemma}

\subsection{A Jackson-type theorem}

The original Jackson Theorem for classical Fourier series asserts that for all $k = 0,1,2,\ldots$ and all functions $g \in C_{\per}^k([-T,T])$, we have
\begin{equation*}
E(g;\mathcal{T}_N) = \mathcal{O}(N^{-k}) \, \omega_{\per}\left(g^{(k)};\frac{1}{N}\right),
\end{equation*}
where the constant in the big $\mathcal{O}$ depends on $k$ and $T$ \cite[Thm.~1.IV]{jackson1930theory}.

There is also a polynomial version of Jackson's Theorem, which states that for all $k = 0,1,2,\ldots$ and all functions $h \in C^k([-1,1])$, we have
\begin{equation}\label{eqn:Jacksonforpolys}
E(h;\mathcal{P}_N) = \mathcal{O}(N^{-k})\, \omega\left(h^{(k)};\frac{1}{N}\right),
\end{equation}
where the constant in the big $\mathcal{O}$ depends only on $k$ \cite[Thm.~1.VIII]{jackson1930theory}. We prove a version of Jackson's Theorem for Fourier extensions.
\begin{theorem}[Jackson-type]\label{thm:Jacksonmain}
  For all $k = 0,1,2,\ldots$ and all functions $f \in C^k([-1,1])$,
  \begin{equation*}
  E(f;\Hspace_N) = \mathcal{O}(N^{-k}) \, \omega\left(f^{(k)};\frac{1}{N}\right),
  \end{equation*}
 where the constant in the big $\mathcal{O}$ depends only on $k$ and $T$.
\end{theorem}

\begin{lemma}[Periodic extension]\label{lem:extension}
  Let $f\in C^k([-1,1])$. Then $f$ can be extended to a function $g\in C_{\per}^k([-T,T])$ such that \[\omega_{\per}(g^{(k)};\delta) \leq \frac{T}{T-1}\omega(f^{(k)};\delta).\]
  \begin{proof}
    First let $k=0$. Define the function $g \in C_{\per}([-T,T])$ such that for $x\in[-1,1]$, $g(x) = f(x)$ and for $x \in[-T,T]\backslash[-1,1]$, $g(x)$ is the the linear function which interpolates $f$ at $\{-1,1\}$. We distinguish between 4 different cases for points $x,y \in [-T,T]$ such that $d_T(x,y) \leq \delta$: $(i)$ if $x,y\in[-1,1]$, then \[|g(x)-g(y)|=|f(x)-f(y)| \leq \omega(f;\delta);\] $(ii)$ if $x,y\in[-T,T]\backslash[-1,1]$, then since $g$ is linear in this region,\[ |g(x)-g(y)|\leq\frac{|f(1)-f(-1)|}{2(T-1)}\delta;\] $(iii)$ if $x\in[-1,1],y\in[-T,T]\backslash[-1,1]$ then \[ |g(x)-g(y)|\leq |f(\xi)-f(x)|+|g(y)-g(\xi)|\leq \omega(f,\delta)+\frac{|f(1)-f(-1)|}{2(T-1)}\delta ,\] where $\xi$ is the closest of the endpoints to $x$; and (iv) if $x\in[-T,T]\backslash[-1,1],y\in[-T,T]$ the bound is similar to the previous one. Now it remains to bound $|f(1)-f(-1)|$ in terms of $\omega(f;\delta)$. For any positive integer $m$, we can use a telescoping sum,
    \begin{align*}
    |f(1)-f(-1)|\leq \sum_{k=0}^{2m-1}\left|f\left(-1+\frac{k}{m}\right)-f\left(-1+\frac{k+1}{m}\right)\right|\leq 2m\omega\left(f;\frac{1}{m}\right)
    \end{align*}
    It suffices to take $m>1/\delta$ to show that $ \frac{|f(1)-f(-1)|}{2(T-1)}\delta\leq \frac{1}{T-1}\omega(f;\delta)$. Combining all four cases demonstrates $\omega_{\per}(g,\delta)\leq \frac{T}{T-1}\omega(f,\delta)$. 
    
    Now let $k>0$ and choose as extension of $f$ the $2(k+1)$th degree Hermite interpolant in the points $x=1$, and $x=-1$; then $g^{(k)}(x)$ is the linear interpolation between $f^{(k)}(1)$ and $f^{(k)}(-1)$ for $x\in[-T,T]\backslash[-1,1]$. By the case $k=0$ proved above, $\omega_{\per}(g^{(k)};\delta)\leq \frac{T}{T-1}\omega(f^{(k)};\delta)$.
  \end{proof}
\end{lemma}

\begin{proof}[Proof of Theorem \ref{thm:Jacksonmain}]
  Let $f \in C^k([-1,1])$. By Lemma \ref{lem:extension}, this function can be extended to a function $g \in C_{\per}^k([-T,T])$ such that $\omega_{\per}(g^{(k)};\delta)$ is bounded by $\frac{T}{T-1}\omega(f^{(k)};\delta)$. Let $t_N \in \mathcal{T}_N$ be the best uniform approximation to $g$, then (trivially) there exists a function $r_N \in \Hspace_N$ such that $r_N(x) = t_N(x)$ for all $x \in [-1,1]$. Hence,
  \begin{equation*}
  E(f;\Hspace_N) \leq \|f - r_N\|_{L^\infty(-1,1)} \leq \|g - t_N\|_{L^\infty(-T,T)} = E(g;\mathcal{T}_N).
  \end{equation*}
  The original Jackson Theorem can now be used to bound $E(g;\mathcal{T}_N)$:
  \begin{equation*}
  E(g;\mathcal{T}_N) = \mathcal{O}(N^{-k}) \, \omega_{\per}(g^{(k)};\delta) \leq \mathcal{O}(N^{-k}) \,\omega(f^{(k)};\delta).
  \end{equation*}
  This proves the result.
\end{proof}

The combination of Lemma \ref{lem:holdermodulus} and Theorem \ref{thm:Jacksonmain} yields the following useful fact. If $f \in C^{k,\alpha}([-1,1])$ for $k \geq 0$ and $\alpha \in [0,1]$, then
\begin{equation*}
E(f;\Hspace_N) = \mathcal{O}(N^{-k-\alpha}) \, |f^{(k)}|_{C^\alpha([-1,1])}.
\end{equation*}
In the sequel we will see that this is not actually tight, in the sense that functions in $C^{k,\alpha}([-1,1])$ can see a decay of best approximation error with a rate faster than $N^{-k-\alpha}$. This is in contrast to the situation for classical Fourier series in which it is indeed tight (see Theorem \ref{thm:JacksonBernstein}).

\subsection{A Bernstein-type theorem}

While Jackson-type theorems bound the best approximation error functional by powers of $N$ and moduli of continuity of derivatives, Bernstein-type theorems attempt to do the opposite. 

Bernstein-type theorems follow from Bernstein inequalities. For classical Fourier series, the Bernstein inequality is,
\begin{equation}\label{eqn:BernsteinFourier}
\|t_N'\|_\infty \leq \frac{\pi}{\T}n \|t_N\|_\infty,
\end{equation}
for all $t_N \in \mathcal{T}_N$, where $N = 2n+1$ \cite[Ch.~4, Th.~2.4]{devore1993constructive}. Equality holds when $t_N(x) \propto e^{\pm\frac{i\pi}{\T}nx}$. From Bernstein's inequality it is possible to show that there exists $C_T >0$ such that \cite[Ch.~7, Thm.~3.1]{devore1993constructive},
\begin{equation*}
\omega_{\per}\left(g;\frac{1}{N}\right) \leq \frac{C_T}{n}\sum_{\substack{k=3 \\ k \text{ odd}}}^N E(g;\mathcal{T}_{k}).
\end{equation*}
Now, this is not precisely a converse to Jackson's Theorem, but it implies the following tightness property.
\begin{theorem}[Jackson--Bernstein {\cite[Ch.~7, Thm.~3.3]{devore1993constructive}}]\label{thm:JacksonBernstein}
  Let $g \in C_{\per}([-T,T])$ and $\alpha \in (0,1)$. It holds that
  \begin{equation*}
  E(g;\mathcal{T}_N) = \mathcal{O}(N^{-\alpha}) \iff \omega_{\per}\left(g;\delta \right) = \mathcal{O}(\delta^{\alpha}).
  \end{equation*}
\end{theorem}

The direct analogue of Theorem \ref{thm:JacksonBernstein} for best uniform approximation by algebraic polynomials in $C([-1,1])$ is not true. Indeed, consider the function $h(x) = (1-x^2)^\alpha$, whose modulus of continuity satisfies $\omega(h;\delta) = \mathcal{O}(\delta^{\alpha})$ by Lemma \ref{lem:holdermodulus}. Define the function $g(\theta) = h\left(\cos\left(\theta\right)\right) = \left|\sin\left(\theta \right)\right|^{2\alpha}$ for $\theta \in [-\pi,\pi]$. If $\alpha < \frac12$ then $g \in C^{2\alpha}([-\pi,\pi])$, so $E(g;\mathcal{T}_N) = \mathcal{O}(N^{-2\alpha})$ by Theorem \ref{thm:JacksonBernstein}. Furthermore, the best approximations will be even since $g$ is even, so the approximants are in fact polynomials in $\cos(\theta)$. This implies that the best approximations to $h$ are polynomials in $x$, showing that $E(h;\mathcal{P}_N) = \mathcal{O}(N^{-2\alpha})$, twice as good as would be expected from Jackson's Theorem for algebraic polynomials (equation \eqref{eqn:Jacksonforpolys}).

 It was only in the late twentieth century that characterizations of functions $h\in C([-1,1])$ for which $E(h;\mathcal{P}_N) = \mathcal{O}(N^{-\alpha})$ were developed \cite[Ch.~8]{devore1993constructive}. The key insight is to use \emph{weighted} moduli of continuity. The weighted modulus of continuity with weight $\phi : [-1,1] \to [0,\infty)$ for a function $f \in C([-1,1])$ is defined as
\begin{equation*}
\omega_{\phi}(f;\delta) = \sup_{\substack{x\pm h \in[-1,1]\\0 \leq h < \phi(x)\delta}} |f(x+h) - f(x-h)|.
\end{equation*}
Taking the weight $\phi(x) = \frac12$ returns the standard modulus of continuity in equation \eqref{eqn:modulusofcty}. 

It turns out that if this weighted modulus of continuity is used with $\phi(x) = \sqrt{1-x^2}$, then there is a direct analogue of Theorem \ref{thm:JacksonBernstein} for best uniform approximation by algebraic polynomials.
\begin{theorem}[Ditzian--Totik {\cite[Cor.~7.2.5]{ditzian1987moduli}}]\label{thm:Ditzian-Totik}
  Let $h \in C([-1,1])$ and $\alpha \in (0,1)$. It holds that
  \begin{equation*}
  E(h;\mathcal{P}_N) = \mathcal{O}(N^{-\alpha}) \iff \omega_{\phi}\left(h ;\delta \right) = \mathcal{O}(\delta^{\alpha}),
  \end{equation*}
  where $\phi(x) = \sqrt{1-x^2}$.
\end{theorem}
The proof of Theorem \ref{thm:Ditzian-Totik} depends on the Bernstein inequality for algebraic polynomials, which states that for all $p_N \in \mathcal{P}_N$,
\begin{equation}\label{eqn:Bernsteinpolys}
\|\phi \cdot p_N' \|_{L^\infty(-1,1)} \leq N\|p_N\|_{L^\infty(-1,1)},
\end{equation}
where $\phi(x) = \sqrt{1-x^2}$ \cite[Ch.~4, Cor.~1.2]{devore1993constructive}. Compare this with the Bernstein inequality for classical Fourier series (equation \eqref{eqn:BernsteinFourier}). If we wish to remove the factor of $\phi$ in the left-hand-side of equation \eqref{eqn:Bernsteinpolys} then we must change $N$ to $N^2$ on the right-hand-side; this is then Markov's inequality \cite[Ch.~4, Thm.~1.4]{devore1993constructive}.

A Bernstein inequality was proved for Fourier extensions by Videnskii \cite{videnskii1960extremal}, see also \cite[p.~242]{borwein2012polynomials} and \cite[Sec.~2]{nagy2013bernstein}. It states that for all $r_N \in \Hspace_N$,
\begin{equation}\label{eqn:FEBernsteininequality}
\|\phi \cdot r_N'\|_{L^\infty(-1,1)} \leq \frac{\pi}{T} n \|r_N\|_{L^\infty(-1,1)},
\end{equation}
where the weight function is
\begin{equation}\label{eqn:FEweight}
\phi(x) = \frac{\sqrt{\sin\left((1-x)\frac{\pi}{2\T} \right)\sin\left((1+x)\frac{\pi}{2\T} \right)}}{\cos\left(x\frac{\pi}{2\T}\right)}.
\end{equation}
Note that since the sine function is concave in $[0,\pi]$, we have $|\sin\left(\frac{\pi}{2\T}(1\pm x)\right)| \geq \sin\left(\frac{\pi}{2\T}\right) |1\pm x|$ for $x \in [-1,1]$. Also, $|\sin\left(\frac{\pi}{2\T}(1\pm x)\right)| \leq |\frac{\pi}{2\T}(1\pm x)|$ and $\cos\left(x\frac{\pi}{2\T}\right) \in [\cos\left(\frac{\pi}{2\T}\right),1]$ for $x\in[-1,1]$. Therefore,
\begin{equation*}
\sin\left(\frac{\pi}{2\T}\right)\sqrt{1-x^2} \leq \phi(x) \leq \frac{\pi}{2\T \cos\left(\frac{\pi}{2\T}\right)}\sqrt{1-x^2},
\end{equation*}
and we can change equation \eqref{eqn:FEBernsteininequality} to
\begin{equation}\label{eqn:properBernstein}
\|\phi \cdot r_N'\|_{L^\infty(-1,1)} \leq \frac{\pi}{T\sin\left(\frac{\pi}{2\T}\right)} n \|r_N\|_{L^\infty(-1,1)},
\end{equation}
where $\phi(x) = \sqrt{1-x^2}$. Using the Bernstein inequality in equation \eqref{eqn:properBernstein} we can prove a Bernstein-type theorem for Fourier extensions.
\begin{theorem}[Bernstein-type]\label{thm:BernsteinFE}
  There exists a constant $C_T>0$ such that for all $f \in C([-1,1])$, the following holds:
 \begin{equation*}
 \omega_{\phi}\left(f;\frac{1}{N}\right) \leq \frac{C_T}{n}\sum_{\substack{k=3 \\ k \text{ odd}}}^N E(f;\Hspace_k),
 \end{equation*}
 where $\phi(x) = \sqrt{1-x^2}$ and $N = 2n+1$.
 \begin{proof}
   This follows directly from \cite[Ch.~6, Thm.~6.2]{devore1993constructive} and \cite[Ch.~7, Thm.~5.1(b)]{devore1993constructive}, with $r =1$,  $\mu = 1$, $X = L^\infty(-1,1)$, $\Phi_n = \mathcal{H}_N$, and $Y = W_\infty^1(\phi) := \{f \in W^{1,1}(-1,1): \phi \cdot f' \in L^\infty(-1,1)\}$, where $W^{1,1}(-1,1)$ is the Sobolev space of absolutely continuous functions on $(-1,1)$.
 \end{proof}
\end{theorem}

From this Bernstein-type theorem for Fourier extensions, we get one half of an equivalence theorem between best approximation errors and weighted moduli of continuity. For the full equivalence, one must prove a Conjecture \ref{conj:Jackson} below.
\begin{theorem}\label{thm:onewithconjecture}
  Let $f \in C([-1,1])$ and $\alpha \in (0,1)$. It holds that
  \begin{equation*}
  E(f;\Hspace_N) = \mathcal{O}(N^{-\alpha}) \implies \omega_{\phi}\left(f; \delta\right) = \mathcal{O}(\delta^\alpha).
  \end{equation*} 
  If Conjecture \ref{conj:Jackson} is true, then the reverse implication holds too.
  \begin{proof}
    The forward implication follows immediately from Theorem \ref{thm:BernsteinFE}. Suppose now that Conjecture \ref{conj:Jackson} is true. Then we would have
    \begin{equation}\label{eqn:basicJacksonFE}
    E(f;\Hspace_N) \leq \frac{C_T}{n}\|\phi \cdot f' \|_{L^\infty(-1,1)}
    \end{equation}
    for all $f \in C^1([-1,1])$ by setting $f(x) = F\left(e^{\frac{i\pi}{\T}x}\right)$, because $f \in C^1([-1,1])$ if and only if $F \in C^1(A)$, $x \mapsto q_n(e^{\frac{i\pi}{\T}x}) \in \Hspace$, and $|\mu(e^{\frac{i\pi}{\T}x})| \leq \frac{\pi}{\T}\phi(x)$. We wish to extend this to all $f \in W^{1,1}(-1,1)$ such that $\phi \cdot f' \in L^\infty(-1,1)$ by a density argument, where $W^{1,p}(-1,1)$ is the Sobolev space of $L^p(-1,1)$ functions whose weak derivatives lie in $L^p(-1,1)$. For such a function $f$, one can verify that the functions $f_\rho(x) = f(\rho x)$ for $\rho \in (0,1)$ satisfy: $f_\rho \in W^{1,\infty}(-1,1)$, $f_\rho \to f$ in $L^\infty$, and $\|\phi \cdot f_\rho'\|_\infty \leq \|\phi \cdot f'\|_\infty$. For each $\rho$ and $\varepsilon > 0$ there exists $f_{\rho,\varepsilon} \in C^1([-1,1])$ such that $\|f_{\rho,\varepsilon} - f_\rho\|_{W^{1,\infty}} < \varepsilon$ by density of $C^1([-1,1])$ in $W^{1,\infty}(-1,1)$. Therefore there exists $f_\varepsilon \in C^1([-1,1])$ such that $\|f - f_\varepsilon\|_{L^\infty(-1,1)} < \varepsilon$ and $\|\phi \cdot f_\varepsilon'\|_\infty \leq \|\phi \cdot f'\|_\infty + \varepsilon$. Hence $E(f;\Hspace_N) \leq \|f-f_\varepsilon\|_{L^\infty(-1,1)} + E(f_\varepsilon;\Hspace_N) \leq \left(1 + \frac{C_T}{n}\right)\varepsilon + \frac{C_T}{n}\|\phi \cdot f'\|_\infty$. Since $\varepsilon$ is arbitrary, we have the desired inequality. A similar argument may be found in \cite[pp.~280]{devore1993constructive}.
    
    From the above it would follow that there exists a constant $C_T > 0$ such that
    \begin{equation}\label{eqn:JacksonfinalFE}
    E(f;\Hspace_N) \leq C_T \omega_\phi\left(f;\frac{1}{N}\right),
    \end{equation} 
    from equation \eqref{eqn:basicJacksonFE}, \cite[Ch.~6, Thm.~6.2]{devore1993constructive} and \cite[Ch.~7, Thm.~5.1(a)]{devore1993constructive}, with $r =1$,  $\mu = 1$, $X = L^\infty(-1,1)$, $\Phi_n = \mathcal{H}_N$, and $Y = W_\infty^1(\phi) := \{f \in W^{1,1}(-1,1): \phi \cdot f' \in L^\infty(-1,1)\}$. Equation \eqref{eqn:JacksonfinalFE} would imply that if $\omega_{\phi}\left(f; \delta\right) = \mathcal{O}(\delta^\alpha)$ then $E(f;\Hspace_N) = \mathcal{O}(N^{-\alpha})$, as required.
  \end{proof}
\end{theorem}

\begin{conjecture}[Jackson inequality for polynomials on a circular arc]\label{conj:Jackson}
  For any $T > 1$, define the arc on the complex unit circle,
  \begin{equation*}
  A = \left\{e^{i\theta} : \theta \in \left(-\frac{\pi}{\T},\frac{\pi}{\T}\right) \right\}.
  \end{equation*}
  There exists a constant $C_T>0$ such that for all $F \in C^1(A)$ and all $n \in \mathbb{N}$, there exists a polynomial $q_n$ of degree $n$ such that
  \begin{equation*}
  \sup_{z \in A} \left|F(z)-q_n(z) \right| \leq \frac{C_T}{n}  \sup_{z \in A} \left| \mu(z) F'(z) \right|,
  \end{equation*}
  where $\mu(z) = \sqrt{(z-e^{\frac{i\pi}{\T}})(z-e^{-\frac{i\pi}{\T}})}$.
\end{conjecture}

Notice that to approximate $f$ we conjecture that we only need to use positive powers of $z$, which means we do not need to utilize all of the functions in $\Hspace_N$. This is because by Mergelyan's Theorem \cite[Thm.~20.5]{rudin1987real} polynomials are dense in the space $C(A)$. It is not surprising because of the redundant nature of approximation by Fourier extensions.

\section{A localization theorem for Fourier extensions}\label{sec:localization}
The theorem proved in this section is a modification of a theorem of Freud (\cite[Thm.~IV.5.4]{freud1971orthogonal}), which is a localization theorem for orthogonal polynomials on an interval. We, however, are working with the orthonormal basis given in Lemma \ref{lem:HNonb}, and there are some clear differences between the two situations. We show that these differences do not change the statement of the result.

\begin{theorem}[Localization theorem]\label{thm:localization}
  Let $f\in L^2(-1,1)$ be such that $f(x) = 0$ for all $x \in [a,b] \subseteq [-1,1]$. Then $\PN(f) \to 0$ uniformly in all subintervals $[c,d] \subset (a,b)$.
  
  \begin{proof}
    First note that the pointwise error can be written in terms of the prolate kernel discussed in Section \ref{sec:prolatekernel} as,
    \begin{equation*}
    \PN(f)(x) - f(x) = \int_{-1}^1 (f(y)-f(x))K_N(x,y)\d y.
    \end{equation*}
    Let $x \in [c,d] \subset (a,b)$, so that $f(x) = 0$. By the formula for the prolate kernel (Lemma \ref{lem:prolatekernelformula}),
    \begin{equation*}
    \PN(f)(x) - f(x) = \int_{-1}^1 \frac{f(y)}{\sin\left( \frac{\pi}{2\T}(x-y)\right)} \mathrm{Imag}\left(\overline{e^{-\frac{i\pi}{\T}\frac{N}{2}y} \cdot \Pi_N\left(e^{\frac{i\pi}{T}y}\right)} \cdot e^{-\frac{i\pi}{\T}\frac{N}{2}x} \cdot \Pi_N\left(e^{\frac{i\pi}{T}x}\right) \right) \, \d y.
    \end{equation*}
    By expressing the imaginary part as $\frac1{2i}$ times the difference of the complex conjugates, it is easy to see that for this expression to tend to zero as $N \to \infty$, it is sufficient that for any $f$ as in the statement of the theorem, we have
    \begin{equation*}
    \lim_{N\rightarrow\infty} \int_{-1}^1 \frac{f(y)}{\sin\left( \frac{\pi}{2\T}(x-y)\right)} e^{\frac{i\pi}{\T}\frac{N}{2}y} \cdot \overline{\Pi_N\left(e^{\frac{i\pi}{T}y}\right)} \, \d y = 0.
    \end{equation*}
    To prove this we consider the functions
    \begin{equation*}
    g_\xi(y) = \frac{f(y)e^{\frac{i\pi}{2\T}y}}{\sin\left( \frac{\pi}{2\T}(\xi-y)\right)},
    \end{equation*}
    for $\xi \in [c,d]$. It holds that $g_{\xi} \in  L^2(-1,1)$, because $g_{\xi}$ is equal to $0$ inside $[a,b]$ and equal to $f$ (an $L^2(-1,1)$ function) multiplied by a bounded function ($y \mapsto e^{\frac{i\pi}{2\T}y} / \sin\left( \frac{\pi}{2\T}(\xi-y)\right)$) outside of $[a,b]$. 
    
    Let $\varepsilon > 0$. By Proposition \ref{prop:L2convergence}, for any $\xi \in [c,d]$, there exists $K_\xi \in \mathbb{N}$ and a function $h_{K_\xi} \in \Hspace_{K_\xi}$ such that $\left\| g_\xi - h_{K_\xi} \right\|_{L^2(-1,1)} < \eps$.
    A key property of the function $h_{K_\xi}$ is that for $N \geq K_\xi$,
    \begin{equation}\label{eqn:coolorthogonality}
    \int_{-1}^1 h_{K_\xi}(y) e^{\frac{i\pi}{\T}\frac{N-1}{2}y} \cdot \overline{\Pi_N\left(e^{\frac{i\pi}{T}y}\right)} \, \d y = 0,
    \end{equation}
    because $h_{K_\xi}(y) e^{\frac{i\pi}{\T}\frac{N-1}{2}y}$ is a polynomial of degree $\frac{K_\xi-1}{2} + \frac{N-1}{2} \leq N - 1$ in the variable $z = \exp\left(\frac{i\pi}{\T}y\right)$. Now, because the map $x \mapsto g_x$ is a continuous mapping from $[c,d] \to L^2(-1,1)$, there exists an interval $I(\xi)$ such that for all $x \in I(\xi)$, $\left\| g_x - h_{K_\xi} \right\|_{L^2(-1,1)} < \eps$ is still valid. In consequence of the Heine-Borel Compactness Theorem \cite{rudin1987real}, the interval $[c,d]$ will be covered by finitely many of these intervals $I(\xi)$, which we denote, $I(\xi_1), I(\xi_2),\dots,I(\xi_s)$. 
    
    Let $K_\varepsilon$ be an odd integer such that $h_{K_{\xi_i}}\in \Hspace_{K_\varepsilon}$ for $i = 1,\ldots,s$. For an arbitrary $x\in[c,d]$ there is an interval $I(\xi_r)$ such that $x\in I(\xi_r)$ and for $N>K_{\varepsilon}$, we have (using the expression \eqref{eqn:coolorthogonality}),
    
    \begin{eqnarray*}
      \left|\int_{-1}^1 \frac{f(y)}{\sin\left( \frac{\pi}{2\T}(x-y)\right)} e^{\frac{i\pi}{\T}\frac{N}{2}y} \cdot \overline{\Pi_N\left(e^{\frac{i\pi}{T}y}\right)} \, \d y\right| &=& \left|\int_{-1}^1 \left(g_{x}(y) - h_{K_{\xi_r}}(y) \right) e^{\frac{i\pi}{\T}\frac{N-1}{2}y} \overline{\Pi_N\left(e^{\frac{i\pi}{T}y}\right)} \, \d y\right| \\
      &\leq& \left(\int_{-1}^1 |g_{x}(y) - h_{K_{\xi_r}}(y)|^2 \,\d y\right)^{\frac12} \\
      & & \qquad \cdot \left(\int_{-1}^1 \left|e^{\frac{i\pi}{\T}\frac{N-1}{2}y} \overline{\Pi_N\left(e^{\frac{i\pi}{T}y}\right)}\right|^2 \,\d y \right)^{\frac12} \\
      &<& \varepsilon.
    \end{eqnarray*}
    This last line used the normality of the basis for $\Hspace_N$ discussed in Lemma \ref{lem:HNonb}. 
    
    In conclusion, since $\varepsilon$ is arbitrary and the inequality above is valid for all $N > K_{\varepsilon}$, the integral must converge to zero as $N \to \infty$, uniformly with respect to $x \in [c,d]$, as required.
  \end{proof}
\end{theorem}

\section{Numerical experiments}\label{sec:numerical}
% !TEX root = convergenceFEpaper1V4.tex

In this section we provide numerically computed examples of pointwise and uniform convergence of Fourier extensions for functions with various regularity properties. It was discussed in the introduction that the condition number of the linear system for computing the Fourier extension is extremely ill-conditioned, making computation of the exact solution to the Fourier extension practically impossible. To deal with this issue, we used sufficiently high precision floating point arithmetic and we did not take $N$ higher than 129, to ensure that the system could be inverted accurately. The right-hand-side vectors for the computations are computed by quadrature in high precision floating point arithmetic.

In practice, one would compute a fast regularized oversampled interpolation Fourier extension using the algorithm in \cite{matthysen2016fast}, requiring only $\mathcal{O}(N\log^2(N))$ floating point operations. However, we are interested %DH interested
in the exact Fourier extension and want to avoid any artefacts that may be caused by the regularization or discretization of the domain.

In some cases, we compare the convergence rate of Fourier extensions to that of Legendre series, because we predict that the qualitative behaviour of Legendre series will be similar (see Section \ref{sec:discussion}). For the Legendre series approximations we computed the Legendre series coefficients one by one using adaptive quadrature in 64 bit floating point precision. As such, the errors for the Legendre series approximations will stagnate due to numerical error.

\subsection{Analytic and entire functions}

Theorem \ref{thm:analyticuniform} gives an upper bound on the rate of exponential convergence of Fourier extension approximations to analytic functions. The regions of analyticity in the complex plane which dictate the rate are the mapped Bernstein ellipse $\mathcal D(\rho)$, where $\rho > 1$. The theorem is illustrated in Figure \ref{fig:ExponentialConvergence}, where we approximate an entire function and four analytic functions which each have a pole in a different location in the complex plane. All examples exhibit exponential convergence in the uniform norm at a rate which is predicted by Theorem \ref{thm:analyticuniform}. This is also the case for %DH for
the entire function, where the exponential convergence rate is limited by a $T$-dependent upper bound.

\begin{figure}[h!]
	\centering
	\resizebox{.60\columnwidth}{!}{\includegraphics[]{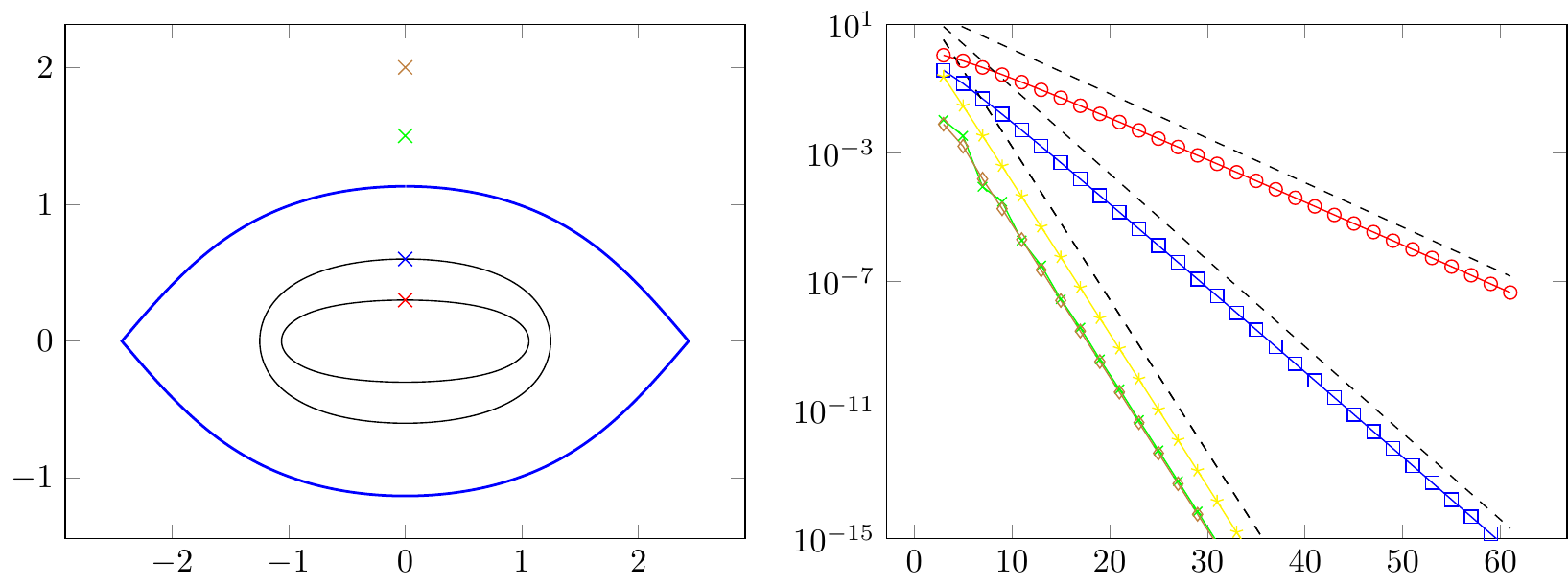}}
	\caption{\label{fig:ExponentialConvergence} We compute Fourier extension approximations to 5 functions: $f(x) = e^x$ (yellow stars) and $f(x) = \frac{1}{x-r}$ for $r=0.3i,~0.6i,~1.5i,~2.0i$ (red circles, blue squares, green crosses and brown diamonds, respectively). The $T$ parameter is $2.43$. Left: the mapped Bernstein ellipses $\mathcal{D}(\rho)$ in the complex plane, for $\rho = 1.891, 3.454, 8.913$. The outermost outline (in blue) encloses the maximal mapped Bernstein ellipse; analyticity outside this largest region does not increase the exponential convergence rate. Right: The $L^\infty(-1,1)$ error against values of $N$ for each of the 5 functions. The black dashed lines indicate the convergence rates predicted by Theorem \ref{thm:analyticuniform}. Color figures online.}
\end{figure}

\subsection{Differentiable functions}

 We investigate Fourier extension approximation of splines of degree $d = 3, 9$ and $15$ on the interval $\left[0,\frac12\right]$, which lie in the H\"older spaces $C^{2,1}\left(\left[0,\frac12\right]\right)$, $C^{8,1}\left(\left[0,\frac12\right]\right)$ and $C^{14,1}\left(\left[0,\frac12\right]\right)$ respectively. By Theorem \ref{thm:Linftyalgebraic}, we expect the pointwise errors to be $\mathcal{O}(N^{-d}\log N)$ in the interior, and $\mathcal{O}(N^{\frac12 - d})$ uniformly over the whole interval.

The spline functions and the pointwise approximation errors for Fourier extensions with various values of $N$ are plotted in Figure \ref{fig:splines}. The rates of convergence predicted by Theorem \ref{thm:Linftyalgebraic} fit reasonably well, sometimes performing slightly better. For comparison, we include the errors for a Legendre series approximation in a dashed line of the same color. See subsection 8.1 for a full discussion comparing convergence of Legendre series and Fourier extensions.

\begin{figure}[h!]
	\centering
	\resizebox{.95\columnwidth}{!}{\input{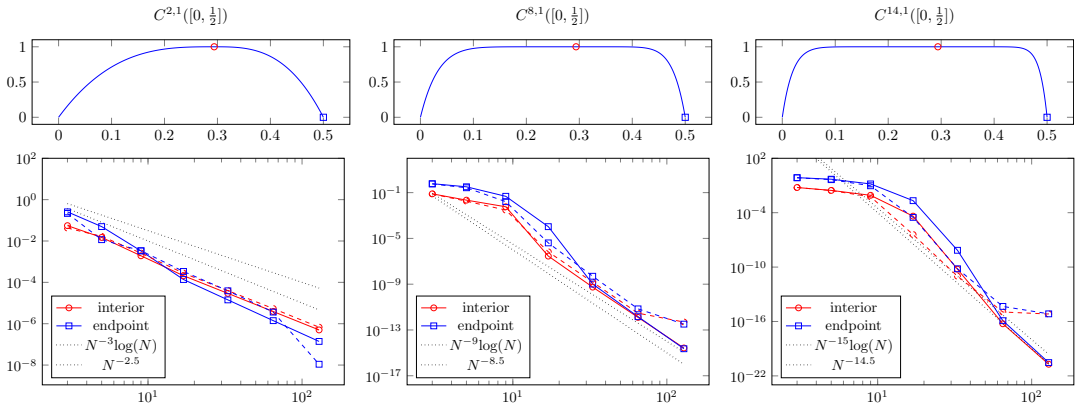}}%
	\\
	\resizebox{.95\columnwidth}{!}{\begin{tikzpicture}[]
\begin{groupplot}[group style={horizontal sep = 1.5cm, group size=3 by 1}]
\nextgroupplot [height = {200}, legend pos = {south west}, ymax = {100.0}, ymode = {log}, xmode = {log}, width = {250}]\addplot+ [mark = {o}, red]coordinates {
(3.0, 0.05526509673788803)
(5.0, 0.014851700440435184)
(9.0, 0.0019318011760717719)
(17.0, 0.0002092978759196738)
(33.0, 2.9264707100282884e-5)
(65.0, 3.881163313426071e-6)
(129.0, 5.001922288660172e-7)
};
\addlegendentry{interior}
\addplot+ [mark = {square}, blue]coordinates {
(3.0, 0.25658753198379586)
(5.0, 0.04972475031581476)
(9.0, 0.0033241932485276802)
(17.0, 0.0001372217011065452)
(33.0, 1.4213643989018907e-5)
(65.0, 1.408577454407209e-6)
(129.0, 1.3921888238506686e-7)
};
\addlegendentry{endpoint}
\addplot+ [mark = {none}, black,dotted]coordinates {
(3.0, 0.37037037037037035)
(5.0, 0.08)
(9.0, 0.013717421124828532)
(17.0, 0.0020354162426216163)
(33.0, 0.0002782647410746584)
(65.0, 3.641329085116067e-5)
(129.0, 4.658336629106499e-6)
};
\addlegendentry{$N^{-3}$log$(N)$}
\addplot+ [mark = {none}, black,dotted]coordinates {
(3.0, 0.6415002990995842)
(5.0, 0.1788854381999832)
(9.0, 0.041152263374485604)
(17.0, 0.008392236160426746)
(33.0, 0.001598509237426059)
(65.0, 0.0002935733363058188)
(129.0, 5.290853352116002e-5)
};
\addlegendentry{$N^{-2.5} $}
\addplot+ [mark = {o}, red,dashed]coordinates {
(3.0, 0.04227490526295408)
(5.0, 0.01671688450284381)
(9.0, 0.0026011940106971387)
(17.0, 0.00027860812664726176)
(33.0, 3.9758610208107115e-5)
(65.0, 5.392917413971077e-6)
(129.0, 6.783000150445417e-7)
};
\addplot+ [mark = {square}, blue,dashed]coordinates {
(3.0, 0.21544367337307546)
(5.0, 0.011599050940272249)
(9.0, 0.0033934290474340843)
(17.0, 0.0003363984245931856)
(33.0, 3.9026673296984364e-5)
(65.0, 3.647680155374392e-6)
(129.0, 1.095094015254936e-8)
};
\nextgroupplot [height = {200}, legend pos = {south west}, ymax = {100.0}, ymode = {log}, xmode = {log}, width = {250}]\addplot+ [mark = {o}, red]coordinates {
(3.0, 0.07821048358894317)
(5.0, 0.022030396439044097)
(9.0, 0.0058045256340345745)
(17.0, 2.8444519371359425e-7)
(33.0, 5.551366483135386e-10)
(65.0, 1.2578023144224415e-12)
(129.0, 2.669642230033547e-15)
};
\addlegendentry{interior}
\addplot+ [mark = {square}, blue]coordinates {
(3.0, 0.6023928974195496)
(5.0, 0.32344783646800546)
(9.0, 0.04751404033289417)
(17.0, 0.00010645141405430719)
(33.0, 1.0260955351179846e-9)
(65.0, 1.4493638095863932e-12)
(129.0, 2.2911891098911248e-15)
};
\addlegendentry{endpoint}
\addplot+ [mark = {none}, black,dotted]coordinates {
(3.0, 0.050805263425290854)
(5.0, 0.0005120000000000001)
(9.0, 2.581174791713197e-6)
(17.0, 8.432565195863826e-9)
(33.0, 2.1546391219452567e-11)
(65.0, 4.828139283063165e-14)
(129.0, 1.0108637138437517e-16)
};
\addlegendentry{$N^{-9}$log$(N)$}
\addplot+ [mark = {none}, black,dotted]coordinates {
(3.0, 0.08799729754452457)
(5.0, 0.0011448668044798925)
(9.0, 7.743524375139592e-6)
(17.0, 3.476835699745383e-8)
(33.0, 1.237745941669622e-10)
(65.0, 3.8925703344740606e-13)
(129.0, 1.1481204762027884e-15)
};
\addlegendentry{$N^{-8.5} $}
\addplot+ [mark = {o}, red,dashed]coordinates {
(3.0, 0.07632306173424364)
(5.0, 0.01783716473050634)
(9.0, 0.002982542416365641)
(17.0, 6.664107412568399e-7)
(33.0, 1.4994335773010903e-9)
(65.0, 2.933764342571976e-12)
(129.0, 4.931610675384945e-13)
};
\addplot+ [mark = {square}, blue,dashed]coordinates {
(3.0, 0.5711879109817688)
(5.0, 0.266220483296742)
(9.0, 0.016180617867821566)
(17.0, 4.0617470030820995e-6)
(33.0, 4.911917095325576e-9)
(65.0, 6.979762333896942e-12)
(129.0, 3.2283812396209775e-13)
};
\nextgroupplot [height = {200}, legend pos = {south west}, ymax = {100.0}, ymode = {log}, xmode = {log}, width = {250}]\addplot+ [mark = {o}, red]coordinates {
(3.0, 0.059833296638478806)
(5.0, 0.028681460994855384)
(9.0, 0.008256792201651438)
(17.0, 4.136017838705343e-5)
(33.0, 5.852882527477902e-11)
(65.0, 5.602754178509705e-17)
(129.0, 1.9195852399862057e-21)
};
\addlegendentry{interior}
\addplot+ [mark = {square}, blue]coordinates {
(3.0, 0.7311342042789644)
(5.0, 0.49306764195152947)
(9.0, 0.14845949339486622)
(17.0, 0.002122524915882734)
(33.0, 7.0824092217055955e-9)
(65.0, 1.2289445497609736e-16)
(129.0, 2.9191055027041243e-21)
};
\addlegendentry{endpoint}
\addplot+ [mark = {none}, black,dotted]coordinates {
(3.0, 6969.171937625632)
(5.0, 3.2768)
(9.0, 0.00048569357496188614)
(17.0, 3.493543693593927e-8)
(33.0, 1.6683643525542651e-12)
(65.0, 6.401763859229619e-17)
(129.0, 2.1935843828486465e-21)
};
\addlegendentry{$N^{-15}$log$(N)$}
\addplot+ [mark = {none}, black,dotted]coordinates {
(3.0, 12070.959882650834)
(5.0, 7.327147548671311)
(9.0, 0.0014570807248856584)
(17.0, 1.4404249656398218e-7)
(33.0, 9.584023540498836e-12)
(65.0, 5.161267027685162e-16)
(129.0, 2.4914329317952645e-20)
};
\addlegendentry{$N^{-14.5} $}
\addplot+ [mark = {o}, red,dashed]coordinates {
(3.0, 0.06014226477255735)
(5.0, 0.027386264303993046)
(9.0, 0.00574847507890941)
(17.0, 3.846237874860492e-7)
(33.0, 9.898304398348046e-12)
(65.0, 1.1102230246251565e-15)
(129.0, 6.661338147750939e-16)
};
\addplot+ [mark = {square}, blue,dashed]coordinates {
(3.0, 0.707508293333333)
(5.0, 0.44023651369597505)
(9.0, 0.09295274270576752)
(17.0, 3.0428873481015478e-5)
(33.0, 6.544193756775093e-11)
(65.0, 4.580816858293056e-15)
(129.0, 6.978712913794951e-16)
};
\end{groupplot}

\end{tikzpicture}}%
	\caption{\label{fig:splines} Above : Plots of plines of degree 3, 9, and 15 in $C^{2,1}$, $C^{8,1}$, and $C^{14,1}$, respectively with an interior point marked using a red circle, and a boundary point marked with a blue square. Below: The pointwise error at an interior point (red circle) and an endpoint (blue square) using Fourier extension with $T=2$ (full lines) and using Legendre series (dashed lines) against the number of degrees of freedom, $N$. The black dotted lines without markers indicate the upper bounds on the algebraic rates of convergence predicted by Theorem \ref{thm:Linftyalgebraic}. Color figures online.}%DH added T=2
\end{figure}

\subsection{Non-differentiable functions}

We investigate the approximation of functions with algebraic singularities, discontinuities, and Dini-Lipschitz continuity. 

Functions with an algebraic singularity at the endpoint are studied in Figure \ref{fig:holder}. We plot the pointwise errors for Fourier extension and Legendre series approximations to $f(x) = x^\alpha$ for $\alpha = \frac34, \frac12$, and $\frac{1}{10}$. These functions lie in the H\"older spaces $C^\alpha\left(\left[0,\frac12\right]\right)$ for their respective values of $\alpha$. 

While Theorem \ref{thm:Linftyalgebraic} guarantees uniform convergence over $[-1,1]$ only for the first function (since for the other two functions $\alpha \leq \frac12$), in our experiments we find that the error of the other two functions converges to zero too. We believe that this discrepancy is related to the weighted moduli of continuity of these functions being more  favourable than the standard moduli (see Section \ref{sec:bestuniform}). Overall, the observed convergence rates are sometimes better than the predicted rates, but when Fourier extensions are compared with Legendre series, we see similar rates of pointwise convergence, especially at the singular point. See subsection 8.1 for a full discussion comparing convergence of Legendre series and Fourier extensions.

\begin{figure}[h!]
	\centering
	\resizebox{.95\columnwidth}{!}{\input{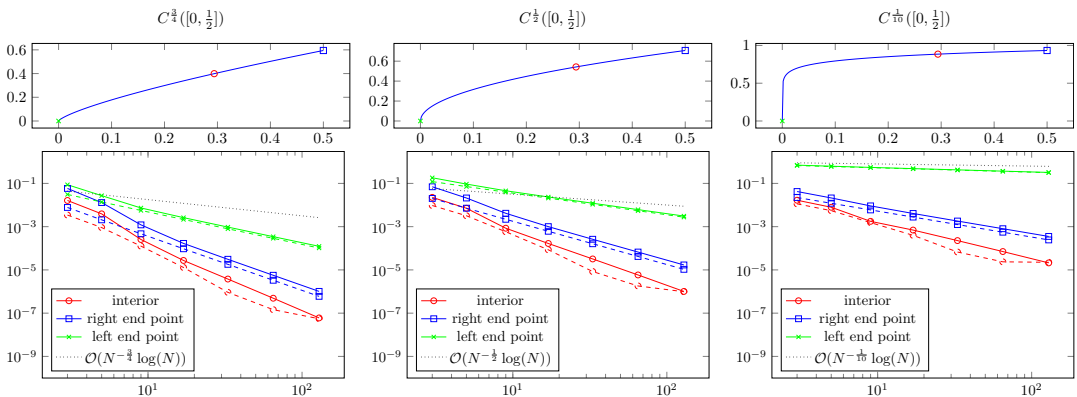}}%
	\\
	\resizebox{.95\columnwidth}{!}{\begin{tikzpicture}[]
\begin{groupplot}[group style={horizontal sep = 1.5cm, group size=3 by 1}]
\nextgroupplot [height = {200}, legend pos = {south west}, ymax = {3.0}, ymode = {log}, xmode = {log}, ymin = {1.0e-10}, width = {250}]\addplot+ [mark = {o}, red]coordinates {
(3.0, 0.016245283219023687)
(5.0, 0.003907186406860123)
(9.0, 0.00025507453410869335)
(17.0, 2.7336120160750247e-5)
(33.0, 3.7782728789690844e-6)
(65.0, 4.907906387532393e-7)
(129.0, 6.027668448348904e-8)
};
\addlegendentry{interior}
\addplot+ [mark = {square}, blue]coordinates {
(3.0, 0.05853008090510997)
(5.0, 0.013024907687321471)
(9.0, 0.0012105443915235035)
(17.0, 0.00016801573293525398)
(33.0, 3.094537371201497e-5)
(65.0, 5.6400877325231246e-6)
(129.0, 1.0145155806624157e-6)
};
\addlegendentry{right end point}
\addplot+ [mark = {x}, green]coordinates {
(3.0, 0.0868538057416897)
(5.0, 0.027207895946750577)
(9.0, 0.007329827323724146)
(17.0, 0.002609881807288721)
(33.0, 0.000955978500780046)
(65.0, 0.0003450214332380034)
(129.0, 0.00012333208058255062)
};
\addlegendentry{left end point}
\addplot+ [mark = {none}, black,dotted]coordinates {
(3.0, 0.043869133765083085)
(5.0, 0.02990697562442441)
(9.0, 0.01924500897298753)
(17.0, 0.011944371675699594)
(33.0, 0.007262974930086152)
(65.0, 0.004368325406821042)
(129.0, 0.0026125085309391134)
};
\addlegendentry{$\mathcal O(N^{-\frac{3}{4}}\log(N))$}
\addplot+ [mark = {o}, red,dashed]coordinates {
(3.0, 0.003441045683934829)
(5.0, 0.0009093837584018272)
(9.0, 0.00012589587768896404)
(17.0, 1.2944759294031272e-5)
(33.0, 9.168077935717278e-7)
(65.0, 1.4494100347706151e-7)
(129.0, 5.522346951947199e-8)
};
\addplot+ [mark = {square}, blue,dashed]coordinates {
(3.0, 0.0077221241084021175)
(5.0, 0.0020674794435021626)
(9.0, 0.0004683702526149247)
(17.0, 9.505020150091248e-5)
(33.0, 1.807947252896458e-5)
(65.0, 3.3191277705180156e-6)
(129.0, 5.981147520595087e-7)
};
\addplot+ [mark = {x}, green,dashed]coordinates {
(3.0, 0.03088849662024623)
(5.0, 0.013783196507253626)
(9.0, 0.00562044332002516)
(17.0, 0.002154471653311092)
(33.0, 0.0007954974755491274)
(65.0, 0.0002876589498905593)
(129.0, 0.00010287785217970335)
};
\nextgroupplot [height = {200}, legend pos = {south west}, ymax = {3.0}, ymode = {log}, xmode = {log}, ymin = {1.0e-10}, width = {250}]\addplot+ [mark = {o}, red]coordinates {
(3.0, 0.022607222437884614)
(5.0, 0.0070916580561896266)
(9.0, 0.0008465140987713085)
(17.0, 0.0001667328380346621)
(33.0, 3.210845167994735e-5)
(65.0, 5.843442404364781e-6)
(129.0, 1.0097969831079067e-6)
};
\addlegendentry{interior}
\addplot+ [mark = {square}, blue]coordinates {
(3.0, 0.06993435775907923)
(5.0, 0.021157762017588625)
(9.0, 0.004153863562692988)
(17.0, 0.0010037611736792768)
(33.0, 0.0002608098586417146)
(65.0, 6.688331175277181e-5)
(129.0, 1.6959190686379354e-5)
};
\addlegendentry{right end point}
\addplot+ [mark = {x}, green]coordinates {
(3.0, 0.18264359506004302)
(5.0, 0.09307354756543247)
(9.0, 0.04569454437020854)
(17.0, 0.02360807895916678)
(33.0, 0.012107476039965202)
(65.0, 0.006139943223551725)
(129.0, 0.0030928781009668146)
};
\addlegendentry{left end point}
\addplot+ [mark = {none}, black,dotted]coordinates {
(3.0, 0.057735026918962574)
(5.0, 0.044721359549995794)
(9.0, 0.03333333333333333)
(17.0, 0.024253562503633298)
(33.0, 0.017407765595569787)
(65.0, 0.012403473458920846)
(129.0, 0.008804509063256239)
};
\addlegendentry{$\mathcal O(N^{-\frac{1}{2}}\log(N))$}
\addplot+ [mark = {o}, red,dashed]coordinates {
(3.0, 0.009499131248269954)
(5.0, 0.0031891758654486457)
(9.0, 0.0005806085458626153)
(17.0, 8.505582676110457e-5)
(33.0, 8.416020449208972e-6)
(65.0, 1.8105177849170317e-6)
(129.0, 9.802811322678906e-7)
};
\addplot+ [mark = {square}, blue,dashed]coordinates {
(3.0, 0.020203050686835144)
(5.0, 0.007142492531841738)
(9.0, 0.0021891848768944566)
(17.0, 0.0006122134609586327)
(33.0, 0.000162366452595486)
(65.0, 4.1842910974421166e-5)
(129.0, 1.0622917469671478e-5)
};
\addplot+ [mark = {x}, green,dashed]coordinates {
(3.0, 0.12121830576829877)
(5.0, 0.07142492784028451)
(9.0, 0.0394053319222809)
(17.0, 0.02081526502347819)
(33.0, 0.01071619968417247)
(65.0, 0.005439605292332683)
(129.0, 0.0027407656571169993)
};
\nextgroupplot [height = {200}, legend pos = {south west}, ymax = {3.0}, ymode = {log}, xmode = {log}, ymin = {1.0e-10}, width = {250}]\addplot+ [mark = {o}, red]coordinates {
(3.0, 0.016700768101535816)
(5.0, 0.007896047815419676)
(9.0, 0.00176176424853406)
(17.0, 0.0006952879222538851)
(33.0, 0.00022665156805525187)
(65.0, 7.068747977730116e-5)
(129.0, 2.10916925876172e-5)
};
\addlegendentry{interior}
\addplot+ [mark = {square}, blue]coordinates {
(3.0, 0.042101586806523046)
(5.0, 0.0212498635571133)
(9.0, 0.009124826537204948)
(17.0, 0.004066491408194989)
(33.0, 0.001817793512057936)
(65.0, 0.0008040133305107636)
(129.0, 0.00035301720907470905)
};
\addlegendentry{right end point}
\addplot+ [mark = {x}, green]coordinates {
(3.0, 0.6993093219141452)
(5.0, 0.6233423902813896)
(9.0, 0.5497721644752085)
(17.0, 0.48315081090020884)
(33.0, 0.42293023871830077)
(65.0, 0.36926403210433445)
(129.0, 0.3219499984417617)
};
\addlegendentry{left end point}
\addplot+ [mark = {none}, black,dotted]coordinates {
(3.0, 0.8959584598407622)
(5.0, 0.8513399225207846)
(9.0, 0.8027415617602307)
(17.0, 0.7532776949250389)
(33.0, 0.7049342406834318)
(65.0, 0.658731853095205)
(129.0, 0.6150933460881862)
};
\addlegendentry{$\mathcal O(N^{-\frac{1}{10}}\log(N))$}
\addplot+ [mark = {o}, red,dashed]coordinates {
(3.0, 0.01134161075179263)
(5.0, 0.005488380391070247)
(9.0, 0.0015396803639289303)
(17.0, 0.00039724913063954315)
(33.0, 6.70765448494981e-5)
(65.0, 2.3644132237765625e-5)
(129.0, 2.2479092054994432e-5)
};
\addplot+ [mark = {square}, blue,dashed]coordinates {
(3.0, 0.022280216562358346)
(5.0, 0.012051135120719758)
(9.0, 0.005948912515274074)
(17.0, 0.0027727132586834458)
(33.0, 0.001250841785308432)
(65.0, 0.0005545183691515465)
(129.0, 0.00024361339786771463)
};
\addplot+ [mark = {x}, green,dashed]coordinates {
(3.0, 0.66840654387555)
(5.0, 0.6025568320407291)
(9.0, 0.5354022588938177)
(17.0, 0.4713615034237672)
(33.0, 0.4127782706824941)
(65.0, 0.36043788495532036)
(129.0, 0.31426316146749256)
};
\end{groupplot}

\end{tikzpicture}}%
	\caption{\label{fig:holder} Above left: $f(x)=x^{3/4}$. Above middle: $f(x)=x^{1/2}$. Above right: $f(x)=x^{1/10}$. Below: The pointwise error at an interior point (red circles), singular endpoint (green crosses), and non-singular endpoint (blue square) using Fourier extension with $T=2$ (full lines) and Legendre series (dashed lines) against the number of degrees of freedom, $N$. The black dotted lines without markers indicate the upper bounds on the algebraic rates of convergence predicted by Theorem \ref{thm:Linftyalgebraic}. Color figures online.}%DH added T=2
\end{figure}

Three functions with a singularity at the interior are shown in Figure \ref{fig:other_asym}. The first has an algebraic singularity: $f(x) = |x-r|^{1/4}$ where $r = 0.29384$ (chosen to avoid any symmetry with respect to the domain). We observe agreement with the expected convergence rate of $\mathcal O(N^{1/4}\log N))$ for the error at interior points. The second function has a jump: 
\begin{equation}\label{eqn:jumpfunction}
f(x) = \begin{cases}
x & \text{ if } x \in \left[0,\frac14\right], \\
1 & \text { if } x \in \left(\frac14,\frac12\right].
\end{cases}
\end{equation}
Even though the function is highly irregular because of the jump, this does not deny convergence at regular points, corroborating the local nature of Theorem \ref{thm:DiniLipschitz}. The last function is uniformly Dini-Lipschitz continuous in $\left[0,\frac12\right]$: 
\begin{equation}\label{eqn:dinilipschitzexample}
f(x) = \begin{cases}
(\log\left(|x-r|\right))^{-2} & \text{ if } x \in \left[0,\frac12\right] \setminus \{r\} \\
0 & \text{ if } x = r,
\end{cases}
\end{equation}
where $r = 0.29384$ (chosen to avoid any symmetry with respect to the domain). In Figure \ref{fig:other_asym}, the expected convergence rate of $\mathcal{O}((\log(N))^{-1})$ of Lemma \ref{lem:superlemma} is present. 
 
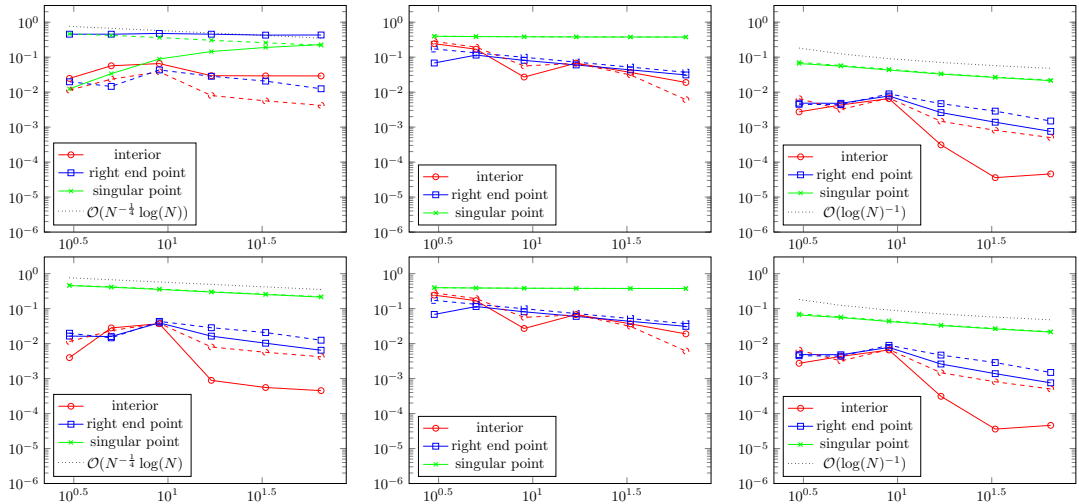
\begin{figure}[h!]
  \centering
  \resizebox{.95\columnwidth}{!}{\begin{tikzpicture}[]
\begin{groupplot}[group style={horizontal sep = 1.5cm, group size=3 by 1}]
\nextgroupplot [height = {200}, legend pos = {south west}, ymax = {3}, ymode = {log}, xmode = {log}, ymin = {1.0e-6}, width = {250}]\addplot+ [mark = {o}, red]coordinates {
	(3.0, 0.024703736671736934)
	(5.0, 0.05678130729862143)
	(9.0, 0.06566436047130246)
	(17.0, 0.029562348782997676)
	(33.0, 0.029228667328362788)
	(65.0, 0.029125369692167235)
};
\addlegendentry{interior}
\addplot+ [mark = {square}, blue]coordinates {
	(3.0, 0.4539619728179141)
	(5.0, 0.45351303546219696)
	(9.0, 0.4767399751170633)
	(17.0, 0.45392755480856123)
	(33.0, 0.42732549580194173)
	(65.0, 0.4311241790565438)
};
\addlegendentry{right end point}
\addplot+ [mark = {x}, green]coordinates {
	(3.0, 0.012431263367141574)
	(5.0, 0.033897418413467925)
	(9.0, 0.0888413906543163)
	(17.0, 0.1450213621766197)
	(33.0, 0.1898533441416946)
	(65.0, 0.22895436471047742)
};
\addlegendentry{singular point}
\addplot+ [mark = {none}, black,dotted]coordinates {
	(3.0, 0.7598356856515925)
	(5.0, 0.668740304976422)
	(9.0, 0.5773502691896257)
	(17.0, 0.4924790605054523)
	(33.0, 0.4172261448611506)
	(65.0, 0.3521856535823236)
};
\addlegendentry{$\mathcal O(N^{-\frac{1}{4}}\log(N))$}
\addplot+ [mark = {o}, red,dashed]coordinates {
	(3.0, 0.011444801398484972)
	(5.0, 0.023392754222046286)
	(9.0, 0.03815194886216755)
	(17.0, 0.008000742024471708)
	(33.0, 0.005635870463252923)
	(65.0, 0.004203936334317315)
};
\addplot+ [mark = {square}, blue,dashed]coordinates {
	(3.0, 0.01979496653640478)
	(5.0, 0.01459643360832763)
	(9.0, 0.04315894710182622)
	(17.0, 0.028440552863825963)
	(33.0, 0.02070552320700103)
	(65.0, 0.012497919155657344)
};
\addplot+ [mark = {x}, green,dashed]coordinates {
	(3.0, 0.4652945617146594)
	(5.0, 0.4188989055193851)
	(9.0, 0.3644781492119915)
	(17.0, 0.30424713498666606)
	(33.0, 0.2589899219009659)
	(65.0, 0.21937012044548793)
};
\nextgroupplot [height = {200}, legend pos = {south west}, ymax = {3}, ymode = {log}, xmode = {log}, ymin = {1.0e-6}, width = {250}]\addplot+ [mark = {o}, red]coordinates {
	(3.0, 0.2426958175642762)
	(5.0, 0.16901563513987758)
	(9.0, 0.027047226195171425)
	(17.0, 0.06764648017794139)
	(33.0, 0.03669384221862281)
	(65.0, 0.01893920318123306)
};
\addlegendentry{interior}
\addplot+ [mark = {square}, blue]coordinates {
	(3.0, 0.06821185226244478)
	(5.0, 0.11460403528062116)
	(9.0, 0.08107280555853713)
	(17.0, 0.05988563121046347)
	(33.0, 0.043388850545186104)
	(65.0, 0.0310612716194205)
};
\addlegendentry{right end point}
\addplot+ [mark = {x}, green]coordinates {
	(3.0, 0.39578953751417534)
	(5.0, 0.3884314184493156)
	(9.0, 0.38265926627220354)
	(17.0, 0.3791282349529578)
	(33.0, 0.37714798448390746)
	(65.0, 0.3760962570671933)
};
\addlegendentry{singular point}
\addplot+ [mark = {o}, red,dashed]coordinates {
	(3.0, 0.281481148)
	(5.0, 0.19341140727711448)
	(9.0, 0.056166765042330424)
	(17.0, 0.06661429555746268)
	(33.0, 0.03109591486196872)
	(65.0, 0.006110295405890809)
};
\addplot+ [mark = {square}, blue,dashed]coordinates {
	(3.0, 0.1718749999999989)
	(5.0, 0.13281250000000344)
	(9.0, 0.09912109375001266)
	(17.0, 0.07227897644046566)
	(33.0, 0.05196670346788912)
	(65.0, 0.037066875925426235)
};
\addplot+ [mark = {x}, green,dashed]coordinates {
	(3.0, 0.3984374999999998)
	(5.0, 0.38964843750000033)
	(9.0, 0.38341140747070335)
	(17.0, 0.37955285335192535)
	(33.0, 0.377376240712215)
	(65.0, 0.3762150294246537)
};
\nextgroupplot [height = {200}, legend pos = {south west}, ymax = {3}, ymode = {log}, xmode = {log}, ymin = {1.0e-6}, width = {250}]\addplot+ [mark = {o}, red]coordinates {
	(3.0, 0.002732688834518953)
	(5.0, 0.0043030735715180405)
	(9.0, 0.00655160228479337)
	(17.0, 0.0003111900774152733)
	(33.0, 3.599775082058022e-5)
	(65.0, 4.600027964148895e-5)
};
\addlegendentry{interior}
\addplot+ [mark = {square}, blue]coordinates {
	(3.0, 0.004848697197578115)
	(5.0, 0.00479899917949927)
	(9.0, 0.007768291855261558)
	(17.0, 0.0026057761480411928)
	(33.0, 0.0013825812248766582)
	(65.0, 0.0007559582373124801)
};
\addlegendentry{right end point}
\addplot+ [mark = {x}, green]coordinates {
	(3.0, 0.0656218766751448)
	(5.0, 0.055443522875961486)
	(9.0, 0.04288543945358269)
	(17.0, 0.03283073145134681)
	(33.0, 0.026235703821956582)
	(65.0, 0.021315143386393173)
};
\addlegendentry{singular point}
\addplot+ [mark = {none}, black,dotted]coordinates {
	(3.0, 0.18204784532536747)
	(5.0, 0.12426698691192238)
	(9.0, 0.09102392266268373)
	(17.0, 0.07059122477295224)
	(33.0, 0.05719993350053487)
	(65.0, 0.047911223149403714)
};
\addlegendentry{$\mathcal O(\log(N)^{-1})$}
\addplot+ [mark = {o}, red,dashed]coordinates {
	(3.0, 0.006499082623767)
	(5.0, 0.0031969553128531336)
	(9.0, 0.006763068303446133)
	(17.0, 0.0014676144816132136)
	(33.0, 0.0008096574426119524)
	(65.0, 0.0005040278525204334)
};
\addplot+ [mark = {square}, blue,dashed]coordinates {
	(3.0, 0.004489338796114489)
	(5.0, 0.004394019535158694)
	(9.0, 0.00892405776705027)
	(17.0, 0.004693913674573752)
	(33.0, 0.0028681442252384004)
	(65.0, 0.0014949158916567618)
};
\addplot+ [mark = {x}, green,dashed]coordinates {
	(3.0, 0.07116574511305483)
	(5.0, 0.057915550695212696)
	(9.0, 0.04512820768206602)
	(17.0, 0.0339452497584122)
	(33.0, 0.027115149322163398)
	(65.0, 0.022015249487918345)
};
\end{groupplot}

\end{tikzpicture}}%
  \\\resizebox{.95\columnwidth}{!}{\begin{tikzpicture}[]
\begin{groupplot}[group style={horizontal sep = 1.5cm, group size=3 by 1}]
\nextgroupplot [height = {200}, legend pos = {south west}, ymax = {3}, ymode = {log}, xmode = {log}, ymin = {1.0e-6}, width = {250}]\addplot+ [mark = {o}, red]coordinates {
(3.0, 0.003969732974359502)
(5.0, 0.028107837652524996)
(9.0, 0.03699089082520602)
(17.0, 0.0008888791369012378)
(33.0, 0.0005551976822663517)
(65.0, 0.0004519000460707983)
};
\addlegendentry{interior}
\addplot+ [mark = {square}, blue]coordinates {
(3.0, 0.016384899818099905)
(5.0, 0.015935962462382738)
(9.0, 0.039162902117249085)
(17.0, 0.01635048180874701)
(33.0, 0.01025157719787246)
(65.0, 0.006452893943270408)
};
\addlegendentry{right end point}
\addplot+ [mark = {x}, green]coordinates {
(3.0, 0.45484536771365713)
(5.0, 0.40851668593304763)
(9.0, 0.3535727136921992)
(17.0, 0.29739274216989586)
(33.0, 0.25256076020482093)
(65.0, 0.21345973963603812)
};
\addlegendentry{singular point}
\addplot+ [mark = {none}, black,dotted]coordinates {
(3.0, 0.7598356856515925)
(5.0, 0.668740304976422)
(9.0, 0.5773502691896257)
(17.0, 0.4924790605054523)
(33.0, 0.4172261448611506)
(65.0, 0.3521856535823236)
};
\addlegendentry{$\mathcal O(N^{-\frac{1}{4}}\log(N)$}
\addplot+ [mark = {o}, red,dashed]coordinates {
(3.0, 0.011444801398484972)
(5.0, 0.023392754222046286)
(9.0, 0.03815194886216755)
(17.0, 0.008000742024471708)
(33.0, 0.005635870463252923)
(65.0, 0.004203936334317315)
};
\addplot+ [mark = {square}, blue,dashed]coordinates {
(3.0, 0.01979496653640478)
(5.0, 0.01459643360832763)
(9.0, 0.04315894710182622)
(17.0, 0.028440552863825963)
(33.0, 0.02070552320700103)
(65.0, 0.012497919155657344)
};
\addplot+ [mark = {x}, green,dashed]coordinates {
(3.0, 0.4652945617146594)
(5.0, 0.4188989055193851)
(9.0, 0.3644781492119915)
(17.0, 0.30424713498666606)
(33.0, 0.2589899219009659)
(65.0, 0.21937012044548793)
};
\nextgroupplot [height = {200}, legend pos = {south west}, ymax = {3}, ymode = {log}, xmode = {log}, ymin = {1.0e-6}, width = {250}]\addplot+ [mark = {o}, red]coordinates {
(3.0, 0.2426958175642762)
(5.0, 0.16901563513987758)
(9.0, 0.027047226195171425)
(17.0, 0.06764648017794139)
(33.0, 0.03669384221862281)
(65.0, 0.01893920318123306)
};
\addlegendentry{interior}
\addplot+ [mark = {square}, blue]coordinates {
(3.0, 0.06821185226244478)
(5.0, 0.11460403528062116)
(9.0, 0.08107280555853713)
(17.0, 0.05988563121046347)
(33.0, 0.043388850545186104)
(65.0, 0.0310612716194205)
};
\addlegendentry{right end point}
\addplot+ [mark = {x}, green]coordinates {
(3.0, 0.39578953751417534)
(5.0, 0.3884314184493156)
(9.0, 0.38265926627220354)
(17.0, 0.3791282349529578)
(33.0, 0.37714798448390746)
(65.0, 0.3760962570671933)
};
\addlegendentry{singular point}
\addplot+ [mark = {o}, red,dashed]coordinates {
(3.0, 0.281481148)
(5.0, 0.19341140727711448)
(9.0, 0.056166765042330424)
(17.0, 0.06661429555746268)
(33.0, 0.03109591486196872)
(65.0, 0.006110295405890809)
};
\addplot+ [mark = {square}, blue,dashed]coordinates {
(3.0, 0.1718749999999989)
(5.0, 0.13281250000000344)
(9.0, 0.09912109375001266)
(17.0, 0.07227897644046566)
(33.0, 0.05196670346788912)
(65.0, 0.037066875925426235)
};
\addplot+ [mark = {x}, green,dashed]coordinates {
(3.0, 0.3984374999999998)
(5.0, 0.38964843750000033)
(9.0, 0.38341140747070335)
(17.0, 0.37955285335192535)
(33.0, 0.377376240712215)
(65.0, 0.3762150294246537)
};
\nextgroupplot [height = {200}, legend pos = {south west}, ymax = {3}, ymode = {log}, xmode = {log}, ymin = {1.0e-6}, width = {250}]\addplot+ [mark = {o}, red]coordinates {
(3.0, 0.002732688834518953)
(5.0, 0.0043030735715180405)
(9.0, 0.00655160228479337)
(17.0, 0.0003111900774152733)
(33.0, 3.599775082058022e-5)
(65.0, 4.600027964148867e-5)
};
\addlegendentry{interior}
\addplot+ [mark = {square}, blue]coordinates {
(3.0, 0.004848697197578115)
(5.0, 0.00479899917949927)
(9.0, 0.007768291855261558)
(17.0, 0.0026057761480411928)
(33.0, 0.0013825812248766582)
(65.0, 0.0007559582373124764)
};
\addlegendentry{right end point}
\addplot+ [mark = {x}, green]coordinates {
(3.0, 0.0656218766751448)
(5.0, 0.055443522875961486)
(9.0, 0.04288543945358269)
(17.0, 0.03283073145134681)
(33.0, 0.026235703821956582)
(65.0, 0.021315143386393173)
};
\addlegendentry{singular point}
\addplot+ [mark = {none}, black,dotted]coordinates {
(3.0, 0.18204784532536747)
(5.0, 0.12426698691192238)
(9.0, 0.09102392266268373)
(17.0, 0.07059122477295224)
(33.0, 0.05719993350053487)
(65.0, 0.047911223149403714)
};
\addlegendentry{$\mathcal O(\log(N)^{-1})$}
\addplot+ [mark = {o}, red,dashed]coordinates {
(3.0, 0.006499082623767)
(5.0, 0.0031969553128531336)
(9.0, 0.006763068303446133)
(17.0, 0.0014676144816132136)
(33.0, 0.0008096574426119524)
(65.0, 0.0005040278525204334)
};
\addplot+ [mark = {square}, blue,dashed]coordinates {
(3.0, 0.004489338796114489)
(5.0, 0.004394019535158694)
(9.0, 0.00892405776705027)
(17.0, 0.004693913674573752)
(33.0, 0.0028681442252384004)
(65.0, 0.0014949158916567618)
};
\addplot+ [mark = {x}, green,dashed]coordinates {
(3.0, 0.07116574511305483)
(5.0, 0.057915550695212696)
(9.0, 0.04512820768206602)
(17.0, 0.0339452497584122)
(33.0, 0.027115149322163398)
(65.0, 0.022015249487918345)
};
\end{groupplot}

\end{tikzpicture}}%
  \caption{\label{fig:other_asym} Above left: $f(x)=|x-r|^{1/4}$ with $r = 0.29384$. Above middle: function with a jump, given in equation \eqref{eqn:jumpfunction}. Above right: Dini-Lipschitz continuous function given in equation \eqref{eqn:dinilipschitzexample}. It has a strong cusp at $x = 0.29384$. Below: The pointwise error at an interior point (red circles), singular interior point (green crosses), and endpoint (blue squares) using Fourier extension with $T=2$ (full lines) and Legendre series (dashed lines) against the number of degrees of freedom, $N$. The black dotted lines without markers in the bottom left plot indicate the upper bounds on the algebraic rates of convergence predicted by Theorem \ref{thm:Linftyalgebraic}. The black dotted line without markers in the bottom right plot indicates the rate of convergence predicted by Lemma \ref{lem:superlemma}. Color figures online.}%DH added T=2
\end{figure}

In all three cases, we compared the convergence of Fourier extension approximations and Legendre series. While there is sometimes a mismatch between the pessimistic prediction of Theorem \ref{thm:Linftyalgebraic} and Lemma \ref{lem:superlemma} for the convergence rates (see Section \ref{sec:bestuniform}), when we compare Fourier extensions and Legendre series, we observe agreement. See subsection 8.1 for a full discussion comparing convergence of Legendre series and Fourier extensions.

\section{Discussion}\label{sec:discussion}

We proved pointwise and uniform convergence results for Fourier extension approximations of functions in H\"older spaces and with local uniform Dini--Lipschitz conditions. This was achieved by proving upper bounds on the associated Lebesgue function and the decay rate of best uniform approximation error for Fourier extensions, then appealing to Lebesgue's Lemma.

\subsection{Comparison to Legendre series}

For a function $f \in L^2(-1,1)$, let us compare the Fourier extension approximant, $f_N$, to the Legendre series approximant,
\begin{equation*}
f^{\mathrm{L}}_N(x) = \sum_{k=0}^{N-1} a_k p^{L}_k(x), \qquad a_k = \frac12\int_{-1}^1 f(x) p^L_k(x) \, \d x,
\end{equation*}
where $p^L_k$ is the $k$th Legendre polynomial normalized so that $\frac12\int_{-1}^1 p^{L}_k(x)^2 \,\d x = 1$. 

The Lebesgue function of this approximation scheme is $\mathcal{O}(\log N)$ for $x \in [a,b] \subset (-1,1)$ and $\mathcal{O}(N^{\frac12})$ uniformly for $x \in [-1,1]$ \cite[Ch.~1]{jackson1930theory}, \cite{gronwall1913laplacesche}, which is precisely the same as the Lebesgue function for Fourier extensions (see Theorem \ref{thm:lebesguefunction}). Best uniform approximation by Fourier extensions was compared to best uniform approximation by algebraic polynomials in Section \ref{sec:bestuniform}. For any $f \in C^{k}([-1,1])$ for $k \in \mathbb{Z}_{\geq 0}$, we have
\begin{equation*}
E(f;\Hspace_N) = \mathcal{O}(N^{-k})\omega\left(f^{(k)};\frac{1}{N}\right) \text{ and } E(f;\mathcal{P}_N) = \mathcal{O}(N^{-k})\omega\left(f^{(k)};\frac{1}{N}\right).
\end{equation*}
It follows that for $C^{k,\alpha}([-1,1])$ functions, the statement of Theorem \ref{thm:Linftyalgebraic} also applies to Legendre series approximations. The localized convergence result for Dini--Lipschitz functions, Theorem \ref{thm:DiniLipschitz} also also applies to Legendre series \cite[Thm.~IV.5.6]{freud1971orthogonal}. Some of the experiments in Section \ref{sec:numerical} demonstrate these similarities.

Theorem \ref{thm:analyticexactsobolev} on exponential convergence differs from the exponential convergence results for Legendre series in two ways. First, the region in the complex plane which determines the rate of exponential convergence is determined by Bernstein ellipses for Legendre series, but by mapped Bernstein ellipses for Fourier extensions. Second, there is an upper limit of $\cot^2\left(\frac{\pi}{4\T}\right)$ for the rate of exponential convergence of Fourier extensions regardless of the region of analyticity, whereas for Legendre series the rate can be arbitrarily fast, and for entire functions the rate of convergence is superexponential \cite{wang2012convergence}.

\subsection{Extensions of this work}

It was mentioned in the introduction that our convergence results will be more applicable if we can extend them to regularized and oversampled interpolation versions of Fourier extensions, because those are the kinds of Fourier extensions for which stable and efficient algorithms have been developed.

Regularized Fourier extensions for a given regularization parameter $\varepsilon > 0$ can be defined as follows. Suppose the matrix $G \in \mathbb{R}^{N\times N}$,
\begin{equation*}
G_{k,j} = \mathrm{sinc}\left( (k-j)\frac{\pi}{\T}\right),
\end{equation*}
has eigendecomposition $G = VSV^*$. Let $S_\varepsilon$ be $S$ but with all entries less than $\varepsilon$ set to $0$. The coefficients $\mathbf{c}^\varepsilon \in \mathbb{C}^N$ of the regularized Fourier extension of $f \in L^2(-1,1)$ are given by
\begin{equation*}
\mathbf{c}^\varepsilon = V S_\varepsilon^\dagger V^* \mathbf{b},
\end{equation*}
where $b_k = \left(\frac\T2\right)^{\frac12}\int_{-1}^1 e^{-\frac{i\pi}{\T}kx}f(x)\,\d x$ \cite{matthysen2016fast}. In other words, the solution is projected onto the eigenvectors of $G$ whose eigenvalues are greater than or equal to $\varepsilon$. These eigenvectors are the Discrete Prolate Spheroidal Sequences (DPSSs), which are the Fourier coefficients of the DPSWFs $\{\xi_{k,N}\}_{k=1}^N$ discussed in Section \ref{sec:prolatekernel} \cite{slepian1978prolate}. The regularized Fourier extension, therefore, finds the best approximation not in $\Hspace_N$, but in the linear space $\Hspace_{N,\varepsilon} \subset \Hspace_N \subset L^2(-1,1)$, where
\begin{equation*}
\Hspace_{N,\varepsilon} = \myspan \left\{ \xi_{k,N} : \lambda_{k,N} \geq \varepsilon \right\}.
\end{equation*}
Therefore, if the Lebesgue function $\Lambda(x;P_{N,\varepsilon})$ (where $P_{N,\varepsilon}$ is the projection operator from $L^2(-1,1)$ to $\Hspace_{N,\varepsilon}$) and best approximation error functional $E(f;\Hspace_{N,\varepsilon})$ can be estimated as in Sections \ref{sec:prolatekernel} and \ref{sec:bestuniform}, then we immediately obtain pointwise convergence results for regularized Fourier extensions by Lebesgue's Lemma. Extensions to the regularized oversampled interpolation version of Fourier extensions can be conducted by considering the analogous quantities for the Periodic Discrete Prolate Spheroidal Sequences (PDPSSs) \cite{xu1984periodic,matthysen2016fast}.

Generalization of this work to the multivariate case would be extremely interesting, because the shape of the domain $\Omega \subset \mathbb{R}^d$ and regularity of its boundary will likely come into play \cite{matthysen2018function}.

\section*{Acknowledgements}

We benefited from useful discussions with Ben Adcock, Arno Kuijlaars,  Walter Van Assche, and Andrew Gibbs. The first author is grateful to FWO Research Foundation Flanders for a postdoctoral fellowship which he enjoyed during the writing of this paper.

\bibliographystyle{abbrv}
\bibliography{convergenceFEbib}

\appendix
\section{Asymptotics of Legendre polynomials on a circular arc}
\label{A:legendre}

Krasovsky derived the asymptotics of polynomials orthogonal on an arc $\{e^{i\theta} : \alpha \leq \theta \leq 2\pi - \alpha\}$ with respect to a positive analytic weight $f_\alpha(\theta)$ by Riemann--Hilbert analysis \cite{krasovsky2004gap,kuijlaars2004riemann,deift1999orthogonal}. We are interested in the case $f_\alpha(\theta) \equiv 1$, the Legendre polynomials on an arc of the unit circle. The following lemma follows Krasovsky's instructions on how to calculate an asymptotic expansion of these polynomials in various regions of the complex plane, where we restrict to the special case of the arc itself.
\begin{lemma}\label{lem:asymArcLeg}
For $\alpha \in (0,\pi)$, let $\{\phi_k(z,\alpha)\}_{k=0}^\infty$ be the polynomials in $z$ with positive leading coefficient, satisfying
\begin{equation*}
\frac{1}{2\pi}\int_{\alpha}^{2\pi-\alpha} \phi_n(e^{i\theta},\alpha) \overline{\phi_m(e^{i\theta},\alpha)}\, \d\theta = \delta_{n,m},
\end{equation*}
for $n,m = 0,1,2,\ldots$. Then there exists $\delta > 0$ such that for $\theta \in [\alpha + \delta,2\pi-\alpha-\delta]$,
\begin{eqnarray}\label{eqn:arclegendre1}
\phi_n(e^{i\theta},\alpha) &=& e^{i\frac{n}{2}\theta} \gamma^{-\frac12}\Bigg( e^{i\frac{\alpha-\pi}{4}}\left( \frac{\sin\left(\frac12(\theta-\alpha)\right)}{\sin\left(\frac12(\theta+\alpha)\right)} \right)^{\frac14} \cos\left(n\tau(\theta) - \frac{\pi}{4}\right) \\
\nonumber & & \qquad \qquad - \quad e^{-i\frac{\alpha-\pi}{4}}\left( \frac{\sin\left(\frac12(\theta+\alpha)\right)}{\sin\left(\frac12(\theta-\alpha)\right)} \right)^{\frac14} \sin\left(n\tau(\theta) - \frac{\pi}{4}\right) \Bigg) + \mathcal{O}(n^{-1})\label{eqn:arclegendre1b}
\end{eqnarray}
and for $\theta \in [\alpha,\alpha + \delta]$,
\begin{eqnarray}\label{eqn:arclegendre2}
\phi_n(e^{i\theta},\alpha) &=& e^{i\frac{n}{2}\theta} \gamma^{-\frac12} \left(\frac{\pi}{2}n\tau(\theta) \right)^{\frac12} \Bigg( e^{-i\frac{\alpha+\pi}{4}}\left( \frac{\sin\left(\frac12(\theta+\alpha)\right)}{\sin\left(\frac12(\theta-\alpha)\right)} \right)^{\frac14} J_0\left(n\tau(\theta)\right) \\
\nonumber& & \qquad \qquad \qquad \quad - \quad e^{i\frac{\alpha+\pi}{4}}\left( \frac{\sin\left(\frac12(\theta-\alpha)\right)}{\sin\left(\frac12(\theta+\alpha)\right)} \right)^{\frac14} J_1\left(n\tau(\theta)\right) \Bigg) + \mathcal{O}(n^{-\frac12})\label{eqn:arclegendre2b}
\end{eqnarray}
where 
\begin{equation*}
\tau(\theta) = \cos^{-1}\left( \frac{\cos\left(\theta/2\right)}{\gamma} \right) \quad \text{ and } \quad \gamma = \cos\left(\frac\alpha2\right).
\end{equation*}
The asymptotics for $\theta \in [2\pi-\alpha-\delta,2\pi-\alpha]$ can be determined using $\phi_n(e^{-i\theta},\alpha) = \overline{\phi_n(e^{i\theta},\alpha)}$.
\begin{proof}
In the notation of \cite{krasovsky2004gap}, the arc on which the polynomials are orthogonal is contained within ``region 1" of the complex plane. The asymptotics of $\phi_n(z,\alpha)$ in region 1 (for $f_\alpha(\theta) =1$) are given by
  \begin{equation}\label{eqn:region1}
  \phi_n(z,\alpha) = z^{\frac{n}{2}}\chi_n \gamma^n \left(\begin{array}{c} 1 \\ 0 \end{array}\right)^T R(z) M(z)\left(\begin{array}{c} (\psi(z)/\sqrt{z})^{n} \\ (\psi(z)/\sqrt{z})^{-n} \end{array} \right),
  \end{equation}
  where $\chi_n$ is the leading coefficient of $\phi_n(z,\alpha)$, $\gamma = \cos(\alpha/2)$, $R(z)$ is a $2\times 2$-matrix-valued function which is analytic and satisfies $R(z) = I + \mathcal{O}(n^{-1})$, $M(z)$ is a $2 \times 2$-matrix-valued analytic function whose expression changes depending on whether $z$ is in a neighbourhood of the endpoints of the arc or not (see below), and $\psi(z)$ is a conformal mapping of the outside of the arc to the outside of the unit circle, given by
  \begin{equation*}
  \psi(z) = \frac{1}{2\gamma}\left( z + 1 + \sqrt{(z-e^{i\alpha})(z-e^{-i\alpha})}\right).
  \end{equation*}
  The branch of the square root which is positive for positive arguments is taken. This is similar to \cite[Eqn.~2.56]{krasovsky2004gap}, which gives the asymptotics of $\phi_n(z,\alpha)$ in subsets of the complex plane outside a fixed neighbourhood of the arc. The job of this Lemma is to unpack this expression and convert to the variable $\theta \in [\alpha,2\pi-\alpha]$ where $z = e^{i\theta}$.
  
  The leading coefficient has asymptotic expression $\chi_n = \gamma^{-n-\frac12}(1+\mathcal{O}(n^{-1}))$ (by \cite[Eq.~2.58]{krasovsky2004gap}), and we have after some algebraic manipulation, for all $\theta \in [\alpha,2\pi-\alpha]$,
  \begin{equation*}
  \psi\left(e^{i\theta} \right) / \sqrt{e^{i\theta}} = \frac{1}{\gamma}\left(\cos\left(\frac{\theta}{2}\right) - i\sqrt{\sin\left(\frac12(\theta+\alpha)\right)\sin\left(\frac12(\theta-\alpha)\right)}\right).
  \end{equation*}
  Since $\left|\psi\left(e^{i\theta} \right) / \sqrt{e^{i\theta}} \right| = 1$ (which can be shown directly or inferred from the conformal mapping definition of $\psi$ above), the function $\tau(\theta)$ defined in the statement of the lemma maps $\theta \in [\alpha,2\pi-\alpha]$ to $\tau \in [0,\pi]$, and provides us with the simple identity, $\psi\left(e^{i\theta} \right) / \sqrt{e^{i\theta}} = e^{-i\tau(\theta)}$. Substituting this into equation \eqref{eqn:region1} we write,
  \begin{eqnarray*}
  \phi_n(e^{i\theta},\alpha) &=& (1+\mathcal{O}(n^{-1}))e^{i\frac{n}{2}\theta} \gamma^{-\frac12} (M_{11}(e^{i\theta})e^{-in\tau(\theta)} + M_{12}(e^{i\theta})e^{in\tau(\theta)}) \\
  & & \qquad + \quad \mathcal{O}(n^{-1})(M_{21}(e^{i\theta})e^{-in\tau(\theta)} + M_{22}(e^{i\theta})e^{in\tau(\theta)})\\
  &=& e^{i\frac{n}{2}\theta} \gamma^{-\frac12} (M_{11}(e^{i\theta})e^{-in\tau(\theta)} + M_{12}(e^{i\theta})e^{in\tau(\theta)}) + \mathcal{O}(n^{-1})e_n(\theta),
  \end{eqnarray*}
  where $e_n(\theta) = \left|M_{11}(e^{i\theta})e^{-in\tau(\theta)} + M_{12}(e^{i\theta})e^{in\tau(\theta)}\right| + \left|M_{21}(e^{i\theta})e^{-in\tau(\theta)} + M_{22}(e^{i\theta})e^{in\tau(\theta)}\right|$.
  
According to \cite[Eq.~2.23]{krasovsky2004gap}, there exists $\delta > 0$ so this asymptotic expression is valid for $\theta \in [\alpha + \delta,2\pi - \alpha - \delta]$ with $M$ set as the function,
  \begin{equation*}
  M(e^{i\theta}) = \frac12\left(\begin{array}{cc} a + a^{-1} & -i(a - a^{-1}) \\ i(a- a^{-1}) & a + a^{-1} 
  \end{array} \right), \qquad a(e^{i\theta}) = e^{i\frac{\alpha}{4}}\left( \frac{\sin\left(\frac12(\theta-\alpha)\right)}{\sin\left(\frac12(\theta+\alpha)\right)} \right)^{\frac14},
  \end{equation*}
  and for $\theta \in [\alpha,\alpha + \delta]$ with $M$ set as the function,
  \begin{eqnarray*}
  M(e^{i\theta}) &=& 2^{-\frac52}\left(\begin{array}{cc} a + a^{-1} & -i(a - a^{-1}) \\ i(a- a^{-1}) & a + a^{-1} 
  \end{array} \right) \left( \begin{array}{cc} 1 & -i \\ -i & 1 \end{array}\right) \\
  & & \cdot\left( \begin{array}{cc} \left(-i\pi n \tau\right)^{\frac12} & 0 \\ 0 & \left(-i\pi n \tau\right)^{-\frac12}\end{array}\right) \left( \begin{array}{cc} H_0^{(1)}(-n\tau) & H_0^{(2)}(-n\tau) \\ -i\pi n\tau(H_0^{(1)})'(-n\tau) & -i\pi n\tau(H_0^{(2)})'(-n\tau) \end{array}\right) \\
  & & \qquad \qquad \qquad \cdot\left( \begin{array}{cc} e^{in\tau} & 0 \\ 0 & e^{-in\tau} \end{array}\right),
  \end{eqnarray*}
where $H_\nu^{(j)}$ is the $\nu$th Hankel function of the $j$th kind \cite[Sec.~10.2]{NIST:DLMF}.
  
  For $\theta \in  [\alpha + \delta,2\pi - \alpha - \delta]$, we have $e_n(\theta) = \mathcal{O}(\delta^{-\frac14}) = \mathcal{O}(1)$. Therefore, grouping terms to convert exponentials into trigonometric functions we obtain  (\ref{eqn:arclegendre1}-\ref{eqn:arclegendre1b}).
  
 For $\theta \in [\alpha,\alpha+\delta]$, we can simplify the formula for $M$ and obtain,
 \begin{eqnarray*}
 M_{11}(e^{i\theta})e^{-in\tau(\theta)} &=& 2^{-\frac32}(-i\pi n \tau)^{\frac12}\left(a^{-1}H_0^{(1)}(-n\tau) -i a (H_0^{(1)})'(-n\tau) \right) \\
 M_{12}(e^{i\theta})e^{in\tau(\theta)} &=& 2^{-\frac32}(-i\pi n \tau)^{\frac12}\left(a^{-1}H_0^{(2)}(-n\tau) -i a (H_0^{(2)})'(-n\tau) \right) \\
  M_{21}(e^{i\theta})e^{-in\tau(\theta)} &=& 2^{-\frac32}(-i\pi n \tau)^{\frac12}\left(-ia^{-1}H_0^{(1)}(-n\tau) + a (H_0^{(1)})'(-n\tau) \right) \\
 M_{22}(e^{i\theta})e^{in\tau(\theta)} &=& 2^{-\frac32}(-i\pi n \tau)^{\frac12}\left(-ia^{-1}H_0^{(2)}(-n\tau) + a (H_0^{(2)})'(-n\tau) \right)
 \end{eqnarray*}
 Using the fact that $ J_\nu = \frac12(H_{\nu}^{(1)} + H_{\nu}^{(2)})$ \cite[Eq.~10.4.4]{NIST:DLMF}, $J_0' = - J_1$, $J_0(-z) = J_0(z)$, and $J_1(-z) = -J_1(z)$, we obtain equations (\ref{eqn:arclegendre2}-\ref{eqn:arclegendre2b}), sans the remainder term $\mathcal{O}(n^{-1})e_n(\theta)$, which we must show to be $\mathcal{O}(n^{-\frac12})$. Collecting the terms, we have,
 \begin{equation*}
 e_n(\theta) = 2^{-\frac12}\left(-i\pi n\tau \right)^{\frac12}\left(\left| a(e^{i\theta})^{-1}J_0(n\tau) - i a(e^{i\theta})J_1(n\tau)  \right| + \left| ia(e^{i\theta})^{-1}J_0(n\tau) - a(e^{i\theta})J_1(n\tau) \right|\right).
 \end{equation*}
For all $\tau \in \left[0,\frac{\pi}{2}\right]$, $\tau^2 \leq \frac{\pi^2}{4}(1-\cos(\tau))$, and $\cos(\tau) = \cos(\theta/2)/\gamma$, so $\tau^2 \leq \frac{\pi^2}{4\gamma}(\gamma-\cos(\theta/2)) = \frac{\pi^2}{2\gamma}\sin\left(\frac14(\theta+\alpha)\right)\sin\left(\frac14(\theta-\alpha)\right)$. For $\theta \in [\alpha,\pi]$, we have $\sin\left(\frac14(\theta-\alpha)\right) \leq \sin\left(\frac12(\theta-\alpha)\right)$, so we can conclude,
 \begin{equation*}
 \left(\frac{\tau(\theta)^2}{\sin\left(\frac12(\theta-\alpha)\right)}\right)^{\frac14} = \mathcal{O}(1),
 \end{equation*}
 uniformly for all $\theta \in [\alpha,\pi]$. Note also that Bessel functions are uniformly bounded in absolute value by $1$ (see \cite[Eq.~10.14.1]{NIST:DLMF}). This makes it clear that $e_n(\theta) = \mathcal{O}(n^{\frac12})$, as required.
 
 The fact that $\psi_n(e^{-i\theta},\alpha) = \overline{\psi_n(e^{i\theta},\alpha)}$ follows from the fact that the weight satisfies $f(-\theta) = f(\theta)$, so the coefficients of $\psi_n(z,\alpha)$ are real (see \cite[p.~288]{szeg1939orthogonal}).
\end{proof}

\end{lemma}

\end{document}